\documentclass[a4paper,reqno, 11pt]{amsart}  
\usepackage[DIV=12, oneside]{typearea}
\usepackage[utf8]{inputenc}
\usepackage[T1]{fontenc}
\usepackage[english]{babel}
\usepackage[centertags]{amsmath}
\usepackage{amstext,amssymb,amsopn,amsthm}
\usepackage{mathrsfs}
\usepackage{dsfont}
\usepackage{bbm}
\usepackage{thmtools}
\usepackage{graphicx}
\usepackage[titletoc,title]{appendix}
\usepackage{etexcmds}

\usepackage[backgroundcolor=white, bordercolor=blue,
linecolor=blue]{todonotes}
\usepackage[colorlinks=true, linkcolor=black, citecolor=black]{hyperref}
\usepackage{enumitem}
\setlist[enumerate]{itemsep=0mm}

\addto\extrasenglish{}
\addto\extrasenglish{}
\addto\extrasenglish{}

\theoremstyle{plain}
\declaretheorem[title=Theorem, parent=section]{theorem}
\declaretheorem[title=Lemma,sibling=theorem]{lemma}
\declaretheorem[title=Proposition,sibling=theorem]{proposition}
\declaretheorem[title=Corollary,sibling=theorem]{corollary}

\theoremstyle{definition}
\declaretheorem[title=Definition,sibling=theorem]{definition}
\declaretheorem[title=Remark,sibling=theorem]{remark}
\declaretheorem[title=Remark, numbered=no]{remark*}
\declaretheorem[title=Example, sibling=theorem]{example}
\declaretheorem[title=Assumption, numbered=no]{assumption*}
\declaretheorem[title=Assumption]{assumption}

\numberwithin{equation}{section}

\newcommand{\N}{\mathds{N}}
\newcommand{\R}{\mathds{R}}

\newcommand{\Z}{\mathds{Z}}

\def\hmath$#1${\texorpdfstring{{\rmfamily\textit{#1}}}{#1}}

\newcommand{\cE}{\mathcal{E}}
\newcommand{\cEs}{\cE^{K_s}}

\newcommand{\eps}{\varepsilon}

\newcommand{\diag}{\mathrm{diag}}

\newcommand{\Cn}{C^{(n)}}
\newcommand{\Cns}{C^{(n)}_s}
\newcommand{\Cna}{C^{(n)}_a}
\newcommand{\nZd}{n^{-1}\Z^d}
\newcommand{\Xn}{X^{(n)}}
\newcommand{\cEn}{\cE^{(n)}}
\newcommand{\cEns}{\cE^{(n),\Cns}}
\newcommand{\cEna}{\cE^{(n),\Cna}}
\newcommand{\Bn}{B^{(n)}}

\newcommand{\PP}{\mathds{P}}
\newcommand{\EE}{\mathds{E}}
\newcommand{\BIGOP}[1]
{
\mathop{\mathchoice%
{\raise-0.22em\hbox{\huge $#1$}}%
{\raise-0.05em\hbox{\Large $#1$}}{\hbox{\large $#1$}}{#1}}}

\def\Xint#1{\mathchoice
   {\XXint\displaystyle\textstyle{#1}}%
   {\XXint\textstyle\scriptstyle{#1}}%
   {\XXint\scriptstyle\scriptscriptstyle{#1}}%
   {\XXint\scriptscriptstyle\scriptscriptstyle{#1}}%
   \!\int}
\def\XXint#1#2#3{{\setbox0=\hbox{$#1{#2#3}{\int}$}
     \vcenter{\hbox{$#2#3$}}\kern-.5\wd0}}

\def\dashint{\Xint-}
\newcommand{\BIGboxplus}{\mathop{\mathchoice%
{\raise-0.35em\hbox{\huge $\boxplus$}}%
{\raise-0.15em\hbox{\Large $\boxplus$}}{\hbox{\large $\boxplus$}}{\boxplus}}}

\DeclareMathOperator{\dist}{dist}

\DeclareMathOperator{\supp}{supp}

\DeclareMathOperator{\pv}{p.v.}

\renewcommand{\Bbb}{\mathbb}

\renewcommand{\d}{\textnormal{d}}

\begin{document}
\allowdisplaybreaks
 \title{Markov Chain Approximations for nonsymmetric Processes}
 
 \author{Marvin Weidner}

\address{Fakult\"{a}t f\"{u}r Mathematik\\Universit\"{a}t Bielefeld\\Postfach 100131\\D-33501 Bielefeld}
\email{mweidner@math.uni-bielefeld.de}


\thanks{Marvin Weidner gratefully acknowledges financial support by the German Research Foundation (GRK 2235 - 282638148). Moreover, he would like to thank Moritz Kassmann for helpful discussions.}


\allowdisplaybreaks

\begin{abstract}
The aim of this article is to prove that diffusion processes in $\R^d$ with a drift  can be approximated by suitable Markov chains on $\nZd$. Moreover, we investigate sufficient conditions on the conductances which guarantee convergence of the associated Markov chains to such Markov processes. Analogous questions are answered for a large class of nonsymmetric jump processes. The proofs of our results rely on regularity estimates for weak solutions to the corresponding nonsymmetric parabolic equations and Dirichlet form techniques.
\end{abstract}

\allowdisplaybreaks

\maketitle

\section{Introduction}

The goal of this article is to establish approximations of nonsymmetric diffusions and jump processes in $\R^d$ by Markov chains on $\nZd$ with generators of the form
\begin{align}
\label{eq:mcgen}
L^{(n)} u(x) = 2n^{\alpha}\sum_{y \in \nZd} (u(y)-u(x))C^{(n)}(x,y),~~ x \in \nZd,
\end{align}
where $\alpha \in (0,2]$. Here, $(\Cn)_n$ is a family of conductances $\Cn : \nZd \times \nZd \to [0,\infty)$, $n \in \N$, that are not necessarily symmetric.
The emphasis of this article lies on the lack of symmetry of the conductances under consideration, which causes the limiting process to possess a drift.

To be precise, in this work we investigate the following two questions:

\begin{itemize}
\item[(i)] 
Under what assumptions on $(\Cn)_n$ do the Markov chains $\Xn$ on $\nZd$ with generators $L^{(n)}$ defined as in \eqref{eq:mcgen} converge weakly towards a Markov process $X$ on $\R^d$ with generator $L$ being of one of the two forms
\begin{align}
\label{eq:mpgenl}
&L u(x) = \partial_i (a_{i,j}(x)\partial_j u(x)) - 2b_i(x)\partial_i u(x),\\
\label{eq:mpgennl}
&L u(x) = 2\pv\int_{\R^d}(u(y)-u(x))K(x,y) \d y,
\end{align}
where $a_{i,j},b_i : \R^d \to \R$ for $i,j \in \{1,\dots,d\}$ with $a_{i,j} = a_{j,i}$, and $K : \R^d \times \R^d \to [0,\infty]$ is a nonsymmetric jumping kernel which satisfies a sector-type condition?\\

\item[(ii)] Let either $(a_{i,j})_{i,j}, (b_i)_{i}$, or $K$ be as above, and $X$ be the corresponding Markov process on $\R^d$ with generator given by \eqref{eq:mpgenl} or \eqref{eq:mpgennl}. Under what assumptions on these objects can we find $(\Cn)_n$ such that the sequence of Markov chains $(\Xn)_n$ with generators given by \eqref{eq:mcgen} converges to $X$?
\end{itemize}

Our main results \autoref{thm:CLTl} and \autoref{thm:CLTnl} answer question (i). (ii) is addressed in \autoref{thm:approxl} and \autoref{thm:approxnl}.
Generally speaking, question (i) asks for conditions under which the Markov chains $\Xn$ on $\nZd$ converge towards some Markov process on $\R^d$ and question (ii) deals with the possibility to approximate a given Markov process on $\R^d$ by a family of Markov chains. Thereby, such approximation yields a scheme for the construction of diffusions, respectively jump processes on $\R^d$. 

Questions (i) and (ii) have a long history and have been answered in various contexts in the symmetric case. Stroock and Varadhan (see \cite{SV79}) provide answers to both types of questions for Markov processes $X$ that are generated by non-divergence form operators. \cite{StZh97} is the first article to investigate problems (i) and (ii) for symmetric divergence form operators of second order. They use the regular Dirichlet forms associated to $\Xn$ in order to show convergence and to identify the limiting process. This is rendered possible by proving a priori heat kernel bounds and uniform in $n$ H\"older estimates for solutions to the heat equation on $\nZd$ using the ideas of De Giorgi-Nash-Moser. Their method allows for symmetric conductances $\Cn$ of bounded range under some continuity condition, yielding diffusion processes in the limit with generators of the form \eqref{eq:limitforml} and $b_i \equiv 0$. Moreover, they provide an explicit construction of approximating Markov chains for a given symmetric diffusion matrix $a_{i,j}$. \cite{BaKu08} extends these ideas, which allows them to deal also with unbounded conductances under a second moment condition.\\
Markov chain approximations of reversible jump processes with generators of the form \eqref{eq:mpgennl} have been considered for the first time in \cite{HuKa07}. Their approach follows the program laid out by \cite{BaKu08} and allows for the approximation of $\alpha$-stable like processes whose jumping kernels are of the form
\begin{align}
\label{eq:introex}
K(x,y) = c(x,y)\vert x-y\vert^{-d-\alpha}, ~~ 0 < \lambda^{-1} \le c(x,y) = c(y,x) \le \lambda, ~~\forall x,y \in \R^d,
\end{align}
for some $\alpha \in (0,2)$ and $\lambda > 0$. Their approach also features jump processes with certain anisotropic jumping kernels that do not satisfy a uniform lower bound but still allow for H\"older estimates of the heat kernel. Moreover, they give a full answer to (ii) for a large class of limiting processes. More general anisotropies can be considered by applying the results of \cite{BKK10}. \cite{Xu13} proves Markov chain approximations for singular stable-like processes, i.e., processes with generators of the form \eqref{eq:limitformnl} but with $K$ being supported only on a $\lambda^d$-null set. Another approach to Markov chain approximations of a large class of symmetric Markov jump processes is developed in \cite{CKK13}. They establish convergence of the finite dimensional distributions by showing a Mosco convergence result for nonsymmetric forms and prove tightness with the help of the Lyons-Zheng decomposition, avoiding the proof of Hölder regularity estimates.\\
Several of the aforementioned results are included in the central limit theorem provided in \cite{BKU10}, where symmetric diffusion processes with jumps are considered without any continuity assumptions on $\Cn$.\\

As opposed to the works mentioned above, in this article we deal with nonsymmetric conductances $\Cn$, which causes the corresponding bilinear forms to be merely regular lower bounded semi-Dirichlet forms. We construct both, nonsymmetric diffusions and jump processes on $\R^d$ via approximation of Markov chains. The corresponding generators in the diffusion case (see \eqref{eq:mpgenl}) might possess drift terms $b$ with $|b|^2 \in L^{\theta}(\R^d)$ for some $\theta \in (\frac{d}{2},\infty]$. The main results are \autoref{thm:CLTl} and \autoref{thm:approxl}. Our results on jump processes (see \autoref{thm:CLTnl} and \autoref{thm:approxnl}) take into account nonsymmetric jumping kernels $K$ satisfying a sector condition.\\
Let us compare the main contributions of this article to existing results from the literature dealing with generators related to nonsymmetric forms as in \eqref{eq:mpgenl} and \eqref{eq:mpgennl}:

For limiting processes corresponding to second order divergence form operators, questions (i) and (ii) have been investigated in \cite{DeKu13}, deriving a priori bounds on the heat kernel and H\"older estimates for a class of centered random walks on $\nZd$ (see \cite{Mat06}). Such Markov chains admit a decomposition into cycles of bounded range and length and are governed by conductances $\Cn$ that are not necessarily symmetric but constant along each cycle. Although this class of Markov chains appears to be very specialized, it turns out that it is rich enough to approximate any given diffusion process with a possibly nonsymmetric diffusion matrix $a_{i,j}$.\\
While \cite{DeKu13} considers nonsymmetric diffusion matrices but does not treat operators with drift terms, our method allows for nonsymmetric contributions of lower order, giving rise to nontrivial drifts but restricting ourselves to symmetric diffusion matrices $a_{i,j}$. However we expect a combination of the techniques from \cite{DeKu13} and those from this article to be possible.

Markov chain approximations of nonsymmetric pure jump processes have been established in \cite{MSS18}. The authors apply an entirely different approach, which is inspired by \cite{CKK13} and not based on the derivation of H\"older estimates. They show convergence of the finite dimensional distributions using Mosco convergence for nonsymmetric forms and prove tightness via a semimartingale approach. The generators of the limiting processes are of the form \eqref{eq:mpgennl} and $K$ may be nonsymmetric and as in \eqref{eq:introex} but as some additional regularity is required for $h \mapsto (K(x,x+h) + K(x+h,x))$, their result is more in the flavor of \cite{SV79} although Dirichlet form techniques are carried out.\\
Our results lie somewhat complementary to \cite{MSS18}. Compared to \cite{MSS18}, we do not have to impose the aforementioned continuity condition on $K$ but on the other hand, \cite{MSS18} does not rely on H\"older estimates and therefore also works in situations where such estimates are not available. We refer to \cite{CKK13} for a discussion of this phenomenon and to \cite{BBCK09} for a related example.

Let us comment on the strategy of our proof. We prove convergence of Markov chains $\Xn$, following the framework constructed in \cite{StZh97}, \cite{BaKu08}, \cite{HuKa07}, \cite{BKU10}, \cite{DeKu13}. However, regularity estimates and upper bounds for exit times are not derived from heat kernel estimates but are shown to follow from weak parabolic Harnack inequalities which can be derived using the same  strategy as in \cite{KaWe22}. The underlying techniques are purely analytic and do not rely on the corresponding stochastic process. We establish exit time estimates and thus tightness of the laws of $\Xn$ by iterating survival estimates that hold uniformly in $\Xn$, adapting the arguments in \cite{Bos20}. Such estimates are a useful tool for the derivation of off-diagonal heat kernel bounds (see \cite{GHL09} \cite{GHL14}, \cite{GHH17}, \cite{GHH18}) for symmetric Markov processes via a purely analytic technique that is based on parabolic maximum principles. Due to the lack of symmetry in our setup, special care is required since the dual semigroup in general does not satisfy the Markov property. In our investigation, we establish a parabolic maximum principle for nonsymmetric operators and come up with another proof of the upper exit time estimate based on the weak Harnack inequality. Moreover, all of the aforementioned results are shown to hold true under abstract unifying assumptions (see \autoref{sec:assumptions}) which allow us to treat the cases  $\alpha =2$ and $\alpha \in (0,2)$ simultaneously. In particular, no truncation of long range conductances is needed.\\
With tightness and H\"older estimates at hand, it remains to prove that all subsequences converge to the same limiting process $X$. We achieve this by investigation of the resolvents of the corresponding Dirichlet forms. Here, we analyze the two cases $\alpha=2$ (with bounded range), and $\alpha \in (0,2)$ separately and also provide answers to question (ii), building upon results in \cite{DeKu13}, \cite{MSS18}.
Our main results are \autoref{thm:CLTl}, \autoref{thm:approxl} (case $\alpha = 2$), as well as \autoref{thm:CLTnl}, \autoref{thm:approxnl} (case $\alpha \in (0,2)$).

We conclude this introduction by emphasizing that most of our methods are robust with respect to degeneracy, unboundedness and irregularity of coefficients. This opens the door to the consideration of homogenization problems for irreversible Markov chains on random media. Quenched invariance principles for symmetric random conductance models with bounded, respectively long range and limiting generators of the form \eqref{eq:mpgenl}, respectively \eqref{eq:mpgennl} can be found in \cite{Bis11}, \cite{ABDH13}, \cite{ADS15}, respectively \cite{BKU10}, \cite{CKK13}, \cite{CKW20c}, \cite{CKW21d}, \cite{FH20}, \cite{Bos20}, \cite{BCKW21}. \\
Moreover, let us point to a related direction of research, namely homogenization problems for local, respectively nonlocal operators with random coefficients. We mention the following articles addressing local operators: \cite{PaVa81}, \cite{Osa83}, \cite{CSW05} and \cite{ChDe16}. More information on this topic can be found in the references therein. Homogenization of symmetric nonlocal operators was e.g., studied in \cite{CCKW21}, \cite{CCKW21b}, \cite{KPZ19}, \cite{FHS19}, \cite{Sch13}, and \cite{ScUe21}. Note that \cite{KPZ19} also contains some results on nonsymmetric nonlocal operators, similar to those in our setup, in case $\alpha < 1$.

\subsection{Outline}
This article is structured as follows: In \autoref{sec:prelim} we construct the Markov chains under consideration and collect some facts about the associated bilinear forms, semigroups and resolvents. We state and discuss the main assumptions of this article in \autoref{sec:assumptions}. \autoref{sec:sol} contains the main technical results needed for convergence including weak Harnack inequalities, H\"older estimates and a weak parabolic maximum principle. The proof of upper exit time estimates, which yield tightness of the laws of $(\Xn)$, is contained in \autoref{sec:tightness}. Convergence of $(\Xn)$, as well as the main results (see \autoref{thm:CLTl}, \autoref{thm:approxl}, \autoref{thm:CLTnl} and \autoref{thm:approxnl}) are stated and proved in \autoref{sec:convergence}.

\section{Preliminaries}
\label{sec:prelim}

The goal of this section is to explain how to associate a family of conductances $(\Cn)$ on $\nZd \times \nZd$ with a family of Markov chains $(\Xn)$ under a suitable assumption on $(\Cn)$ (see \eqref{eq:ass0s}). We choose to introduce $(\Xn)$ as the unique family of Hunt processes associated with the regular lower bounded semi-Dirichlet forms on $L^2(\nZd)$ that are induced by $(\Cn)$. Moreover, we introduce the corresponding heat semigroup and resolvent operators in the sense of \cite{Osh13} and discuss the main assumptions of this article. 

Let us fix $\alpha \in (0,2]$, $n \in \N$ and a family of conductances $\Cn : \nZd \times \nZd \to [0,\infty)$, $n \in \N$ that is not necessarily symmetric, i.e., $\Cn(x,y) \neq \Cn(y,x)$, and satisfies
\begin{align}
\label{eq:ass0}
\Cn(x) := \sum_{y \in \nZd} \Cn(x,y) \le c,
\end{align}
for some $c > 0$ that is independent of $x$, and $\Cn(x,x) = 0$ for every $x \in \nZd$. Under this assumption, the conductances $\Cn$ give rise to the operator $(L^{(n)},\mathcal{D}(L^{(n)}))$ on $L^2(\nZd)$  defined by
\begin{align*}
-L^{(n)} u(x) &= 2n^{\alpha}\sum_{y \in \nZd} (u(x)-u(y))C^{(n)}(x,y),~~ x \in \nZd,~~ u \in \mathcal{D}(L^{(n)}),\\
\mathcal{D}(L^{(n)}) &= \left\lbrace f : \nZd \to \R : \sum_{y \in \nZd} \vert f(y) \vert \Cn(x,y) < \infty, ~~\forall x \in \nZd \right\rbrace,
\end{align*}
where $\alpha \in (0,2]$. Here, $L^2(\nZd) = L^2(\nZd,\mu^{(n)})$,  and $\mu^{(n)}(\{x\}) = n^{-d}$. We denote the scalar product on $L^2(\nZd)$ by $<\cdot,\cdot>_{L^2(\nZd)} = <\cdot,\cdot>$.\\
Via the identity $<-L^{(n)} u,v> = \cEn(u,v)$, we associate this operator with the bilinear form
\begin{align*}
\cEn(u,v) = 2n^{\alpha-d}\sum_{x \in \nZd} \sum_{y \in \nZd} (u(x)-u(y))v(x) \Cn(x,y), ~~ u,v \in L^2(\nZd).
\end{align*}
First, we decompose $\Cn = \Cns + \Cna$ into its symmetric part $\Cns$ and its antisymmetric part $\Cna$ defined by
\begin{align*}
\Cns(x,y) = \frac{\Cn(x,y) + \Cn(y,x)}{2}, \quad \Cna(x,y) = \frac{\Cn(x,y) - \Cn(y,x)}{2}, ~~ x,y \in \nZd.
\end{align*}
Then, we observe that we can rewrite $\cEn$ in terms of $\Cns$ and $\Cna$ as follows:
\begin{align*}
\cEn(u,v) &= 2n^{\alpha-d}\sum_{x \in \nZd}\sum_{y \in \nZd} (u(x)-u(y))v(x) \Cns(x,y)\\
&+2n^{\alpha-d}\sum_{x \in \nZd}\sum_{y \in \nZd} (u(x)-u(y))v(x) \Cna(x,y)\\
&=: \cEns(u,v) + \cEna(u,v).
\end{align*}
Since by construction $\Cns(x,y) = \Cns(y,x)$ and $\Cna(x,y) = -\Cna(y,x)$, it holds:
\begin{align*}
&\cEns(u,v) = n^{\alpha-d}\sum_{x \in \nZd}\sum_{y \in \nZd} (u(x)-u(y))(v(x)-v(y)) \Cns(x,y),\\
&\cEna(u,v) = n^{\alpha-d}\sum_{x \in \nZd}\sum_{y \in \nZd}  (u(x)-u(y))(v(x)+v(y)) \Cna(x,y).
\end{align*}

\begin{lemma}
\label{lemma:markov-welldef}
Assume that for every $n \in \N$, $\Cn(x,x) = 0$ for every $x \in \nZd$, and:
\begin{align}
\label{eq:ass0s}
\sup_{x \in \nZd} \sum_{y \in \nZd} \Cns(x,y) < \infty.
\end{align}
\begin{itemize}
\item[(i)] Then it holds $\cEn(u,v) < \infty$ for every $u,v \in L^2(\nZd)$. Moreover, $(\cEn,L^2(\nZd))$ is a regular lower-bounded semi-Dirichlet form on $L^2(\nZd)$.
\item[(ii)] The generator of $(\cEn,L^2(\nZd))$ is given by $(L^{(n)},L^2(\nZd))$. Moreover, $L^{(n)}$ is bounded in $L^2(\nZd)$.
\end{itemize}
\end{lemma}

\begin{proof}
First of all, note that \eqref{eq:ass0s} implies \eqref{eq:ass0}.
To see that $\cEn(u,v) < \infty$ for every $u,v \in L^2(\nZd)$, we refer the interested reader to \cite{MSS18}. $(\cEn,L^2(\nZd))$ is a regular lower-bounded semi-Dirichlet form since \eqref{eq:ass0s} implies that for every $n \in \N$:
\begin{align}
\label{eq:prelK1}
\sup_{x \in \nZd} \sum_{y \in \nZd} \frac{\vert \Cna(x,y)\vert^2}{\Cns(x,y)} < \infty.
\end{align}
In \cite{ScWa15}, it was proved that $(\cEn,L^2(\nZd))$ satisfies a G\r{a}rding's inequality and the sector condition under \eqref{eq:prelK1}. However, note that at this point the constants might still depend on $n$ (see \autoref{lemma:elementary} for an improved result using \eqref{K1}). For a proof of (ii), we refer to \cite{MSS18}.
\end{proof}

As a regular lower-bounded semi-Dirichlet form, $(\cEn,L^2(\nZd))$ is associated with the so called variable speed random walk $\Xn$, a continuous time Markov chain that jumps from a point $x \in \nZd$ to $y \in \nZd$ with probability $\frac{\Cn(x,y)}{\Cn(x)}$ and waits at $x$ for an exponentially distributed waiting time with parameter $n^{\alpha-d}\Cn(x)$. Note that $\Xn$ is in general non-reversible due to the lack of symmetry of $\Cn$.

\begin{remark}
Note that also $(\cEns, L^2(\nZd))$ is a regular symmetric Dirichlet form and is particular nonnegative definite, i.e., $\cEns(u,u) \ge 0$ for every $u \in L^2(\nZd)$.\\
$(\cEns, L^2(\nZd))$ is a associated with a symmetric Hunt process $Y^{(n)}$. One can construct $Y^{(n)}$ as the reversible continuous time Markov chain jumping from $x \in \nZd$ to $y \in \nZd$ with probability $\Cns(x,y)\left(\sum_{y \in \nZd} \Cns(x,y)\right)^{-1}$ and waits at $x$ for an exponentially distributed waiting time with parameter $n^{\alpha-d}\left(\sum_{y \in \nZd} \Cns(x,y)\right)$.
\end{remark}

Moreover, the following L\'evy system formula holds true.

\begin{lemma}[L\'evy system formula]
\label{lemma:LSformula}
Assume \eqref{eq:ass0s}. Let $f : [0,\infty) \times \nZd \times \nZd$ be nonnegative, measurable, and vanishing on the diagonal. Then, for any $x \in \nZd$ and predictable stopping time $\tau$:
\begin{align*}
\EE^x \left[ \sum_{t \le \tau} f(t,\Xn_{t_-},\Xn_t) \right] = \EE^x \left[ \int_0^{\tau} n^{\alpha} \sum_{y \in \nZd} f(t,\Xn_t,y) \Cn(\Xn_t , y) \d t \right].
\end{align*}
\end{lemma}

\begin{proof}
According to \cite{Osh13}, one can take $N(x,y) = n^{\alpha+d}\Cn(x,y)$, $H_t = t$ as a L\'evy system for $\Xn$. From here, the proof follows along the lines of Lemma 4.7 in \cite{ChKu03}, respectively Lemma 4.1 in \cite{HuKa07}.
\end{proof}

We have seen above that the condition \eqref{eq:ass0s} suffices for the family of conductances $(\Cn)$ to induce a family of regular lower bounded semi-Dirichlet forms, which allows us to define associated Markov chains $\Xn$ on $\nZd$ for every $n \in \N$. However, in order to prove the desired convergence results (see \autoref{thm:CLTl} and \autoref{thm:CLTnl}), it is crucial to impose assumptions that give some control over the behavior of $(\Xn)$ uniformly in $n$. Therefore, in the sequel we will work with the assumptions, which will be introduced in the following section.\\
Let us remark already at this point that \eqref{CTail} implies the existence of a uniform bound in \eqref{eq:ass0s}. Together with \eqref{K1} and \eqref{Sob}, we are able to prove a G\r{a}rding's inequality and a sector condition for $(\cEn,L^2(\nZd))$ with constants that are uniform in $n$ (see \autoref{lemma:elementary}).

\subsection{Main Assumptions}
\label{sec:assumptions}

In the following we list the main assumptions on the conductances that we assume to be in place throughout the remainder of this article. Those are sufficient conditions under which solutions to $\partial_t u - L^{(n)} u = 0$ respectively $-L^{(n)} u = 0$ are locally H\"older continuous and nonnegative supersolutions satisfy a weak Harnack inequality with a constant independent of $n$ (see \autoref{sec:sol}). We point out that all assumptions are formulated in such a way that all the appearing constants are independent of $n$.\\
Similar assumptions, as well as their motivation and a discussion can be found in \cite{KaWe22} in the context of integro-differential operators in $\R^d$ governed by nonsymmetric integral kernels $K : \R^d \times \R^d \to [0,\infty]$.

We define $\Bn_r(x_0) = \lbrace x \in \nZd : \vert x - x_0\vert < r \rbrace \subset \nZd$. In contrast to Euclidean space, $\mu^{(n)}(\Bn_r) \asymp r^d$ does not hold for every $r > 0$ since the upper bound fails as $r \searrow 0$. However, for every $\sigma > 0$ there exist $c_1,c_2 > 0$ such that for every $n \in \N$ and $r > \frac{\sigma}{2n}$: 
\begin{align}
\label{eq:volreg}
c_1 r^d \le \mu^{(n)}(\Bn_r) \le c_2 r^d.
\end{align} 
As such volume regularity property is crucial for the derivation of H\"older estimates, we restrict ourselves to working on balls with large enough radii. This fact is mirrored also in the statements of the assumptions below.

Let $\alpha \in (0,2]$, $\sigma > 0$ and $\theta \in (\frac{d}{\alpha},\infty]$ be fixed. We list the following assumptions on a family of conductances $(\Cn)_n$:

The first two assumptions control the behavior of the antisymmetric part of $\Cn$.

\begin{assumption*}[\textbf{K1}]
\label{ass:K1}
There exist $A > 0$ and a symmetric conductance $J^{(n)} : \nZd \times \nZd \to [0,\infty]$ satisfying \eqref{CTail} such that for every $n \in \N$, $x_0 \in \nZd$, $r > \frac{\sigma}{2n}$ and $v \in L^2(\Bn_{2r}(x_0))$:
\begin{align}
\label{K1}\tag{$\text{K1}$}
\left\Vert n^{\alpha}\sum_{y \in \nZd} \frac{\vert\Cna(\cdot,y)\vert^2}{J^{(n)}(\cdot,y)} \right\Vert_{L^{\theta}(\nZd)} \le A, \quad \cE^{(n),J^{(n)}}_{\Bn_{2r}(x_0)}(v,v) \le A \cEns_{\Bn_{2r}(x_0)}(v,v).
\end{align}
\end{assumption*}

\begin{assumption*}[\textbf{K2}]
\label{ass:K2}
There exist $C > 0$, $D < 1$ and a symmetric conductance $j^{(n)} : \nZd \times \nZd \to [0,\infty]$ such that for every $n \in \N$, $x_0 \in \nZd$, $\frac{\sigma}{2n} < r \le 1$, and every $x,y \in \Bn_2(x_0)$ and every $v \in L^2(\Bn_{2r}(x_0))$:
\begin{align}
\label{K2}\tag{$\text{K2}$}
\Cn(x,y) \ge (1-D) j^{(n)}(x,y), ~~ \cEns_{\Bn_{2r}(x_0)}(v,v) \le C \cE^{(n),j^{(n)}}_{\Bn_{2r}(x_0)}(v,v).
\end{align}
\end{assumption*}

\begin{remark}
\begin{itemize}
\item[(i)] \eqref{K1} is crucial for the validity of the sector condition with a uniform constant (see \autoref{lemma:elementary}). It implies that the antisymmetric part is of lower order. 
\item[(ii)] The range $\theta \in (\frac{d}{\alpha},\infty]$ is natural. It causes the antisymmetric part $\cEna$ to have subcritical scaling. It allows us to approximate operators possessing drifts within this range of integrability.
\item[(iii)] \eqref{K2} can be regarded as an ellipticity assumption on $\Cn$. It ensures that the conductance $\Cns - |\Cna|$ is locally coercive with respect to $\cEns$.
\item[(iv)] In the simplest case, \eqref{K1} and \eqref{K2} hold true with $J^{(n)} = j^{(n)} = \Cns$. Allowing for general kernels $J^{(n)}$, $j^{(n)}$ makes it possible to work with conductances $\Cn$ that are not supported in certain cones of directions (see \cite{KaWe22}).
\end{itemize}
\end{remark}

The following assumption is reminiscent of \eqref{eq:ass0s}, but gives us uniform control in $n$.

\begin{assumption*}[\textbf{C-Tail}]
\label{ass:CTail}
There exist $c,\delta > 0$ such that for every $n \in \N$ it holds
\begin{align}
\label{CTail}\tag{$\text{C-Tail}0$}
&\sup_{x \in \nZd} n^{\alpha}\sum_{y \in \nZd \setminus \Bn_r(x)} \Cns(x,y) \le c r^{-\alpha}, ~~ \forall 0 < r \le 1,\\
\label{CTail2}\tag{$\text{C-Tail}\infty$}
&\sup_{x \in \nZd} n^{\alpha}\sum_{y \in \nZd \setminus \Bn_r(x)} \Cn(x,y) \le c r^{-\delta}, ~~ \forall 1 < r < \infty.
\end{align}
\end{assumption*}

Let us make several remarks on the assumptions introduced above.

\begin{lemma}
Assume that \eqref{CTail} is satisfied with $\alpha \in (0,2]$. Then the following hold true:
\begin{itemize}
\item[(i)] There exists $c > 0$ such that for every $n \in \N$ and $0 < r \le 1$:
\begin{align}
\label{eq:BKUA3}
\sup_{x \in \nZd} n^{\alpha} \sum_{y \in \Bn_r(x)} \vert x-y \vert^{2} \Cns(x,y) \le c r^{2-\alpha}.
\end{align}
\item[(ii)] Let $0 < \rho \le r \le 1$ and $x_0 \in \nZd$. Every function $\tau : \nZd \to [0,1]$ with $\supp(\tau) \subset \Bn_{r+\rho}(x_0)$, $\tau \equiv 1$ in $\Bn_{r}(x_0)$ and $\max_{i = 1,...,d} \Vert\nabla^{(n)}_i \tau\Vert_{L^{\infty}(\nZd)} \le 2\rho^{-1}$ satisfies:
\begin{align}
\label{eq:cutoff}
\sup_{x \in \nZd} n^{\alpha} \sum_{y \in \nZd} (\tau(x)-\tau(y))^2 \Cns(x,y) \le c \rho^{-\alpha},
\end{align}
where $c > 0$ is independent of $\rho,r,x_0,n$, and we write $\nabla^{(n)}_i \tau(x) := n(\tau(x+e_i/n)-\tau(x))$.
\end{itemize}
\end{lemma}

The proof of this result goes via decomposing $\Bn_r(x) = \bigcup_{k=0}^{\infty} \Bn_{2^{-k}r}(x) \setminus \Bn_{2^{-k-1}r}(x)$ and is well-known in the literature.\\ Note that \eqref{eq:cutoff} is a discrete version of assumption $(\text{Cutoff})$ from \cite{KaWe22}.

\begin{remark}[\eqref{CTail}, \eqref{CTail2} for bounded range]
Let $\alpha = 2$ and assume that there exists $C > 0$ such that $\Cn(x,y) = 0$, whenever $\vert x-y \vert > \frac{C}{n}$.

\begin{itemize}
\item[(i)] In this special case, assumption \eqref{CTail2} simplifies significantly since
\begin{align}
\label{eq:bdrangeCtail}
\sup_{x \in \nZd} n^{\alpha} \sum_{y \in \nZd \setminus \Bn_r(x)} \Cn(x,y) = 0,
\end{align}
for every $r > \frac{C}{n}$. As for every $r > 1$, it holds that $r > \frac{C}{n}$ already if $n > C$, we infer that \eqref{CTail2} follows if there exists $c > 0$: 
\begin{align*}
\Cn(x) := \sum_{y \in \nZd} \Cn(x,y) \le c, ~~ \forall x \in \nZd.
\end{align*}

\item[(ii)] \eqref{CTail} follows already if one assumes
\begin{align}
\label{eq:suffCTailbdrange}
\sup_{n\in\N}\sup_{x \in \nZd}\sum_{y \in \nZd} \Cns(x,y) < \infty,
\end{align}
which is due to \eqref{eq:bdrangeCtail} and the fact that for $r \le C/n$:
\begin{align*}
n^{2}\sum_{y \in \nZd \setminus \Bn_r(x)} \Cns(x,y) \le c r^{-2} \sum_{y \in \nZd \setminus \Bn_r(x)} \Cns(x,y) \le c_1 r^{-2}.
\end{align*}
\item[(iii)] Assumption \eqref{eq:suffCTailbdrange} is natural for symmetric Markov chains  (see  \cite{DeKu13}, \cite{BaKu08}, \cite{BKU10}). The uniformity in $n$ usually follows from scaling. Namely, given a symmetric conductance $C_s : \Z^d \times \Z^d \to [0,\infty]$ with $\sup_{x \in \Z^d}\sum_{y \in \Z^d} C_s(x,y) < \infty$, one defines $\Cns(x,y) := C_s(nx,ny)$ for $x,y \in \nZd$, and hence \eqref{eq:suffCTailbdrange} is immediate.
\item[(iv)] In particular, \eqref{eq:suffCTailbdrange} is sufficient for both, \eqref{CTail} and \eqref{CTail2}.
\end{itemize}

\end{remark}

Finally, we require a suitable coercivity assumption. We express coercivity in terms of a Poincar\'e - and a Sobolev inequality. For an investigation of the validity of such inequalities for Markov chains, we refer the interested readers to the monographs \cite{Kum14}, \cite{Bar17}.

\begin{assumption*}[\textbf{Poinc}]
\label{ass:Poinc}
There exists $c > 0$ such that for every $n \in \N$, every ball $\Bn_r \subset \nZd$ with $\frac{\sigma}{2n} < r \le 1$ and every $v \in L^2(\Bn_r)$:
\begin{align}
\label{Poinc}\tag{$\text{Poinc}$}
n^{-d}\sum_{x \in \Bn_r} (v(x) - [v]_{\Bn_r})^2 \le c r^{\alpha} \cEns_{\Bn_r}(v,v),
\end{align}
where $[v]_{\Bn_r} = \mu^{(n)}(\Bn_r)^{-1} n^{-d}\sum_{y \in \Bn_r} v(y)$.
\end{assumption*}

\begin{assumption*}[\textbf{Sob}]
There exists $c > 0$ such that for every $n \in \N$, and every $v \in L^2(\nZd)$:
\begin{align}
\label{Sob}\tag{$\text{Sob}$}
\Vert v^2 \Vert_{L^{\frac{d}{d-\alpha}}(\nZd)} \le c\cEns(v,v).
\end{align}
\end{assumption*}

\begin{remark}
\begin{itemize}
\item[(i)] One can deduce from \eqref{Sob} and \eqref{eq:cutoff} a local Sobolev inequality of the following form:\\
There exists $c > 0$ such that for every $n \in \N$, every $x_0 \in \nZd$ and $\frac{\sigma}{2n} < r \le 1$, $0 < \rho \le r$ and every $v \in L^2(\Bn_{2r}(x_0))$:
\begin{align}
\label{Sob-loc}
\Vert v^2 \Vert_{L^{\frac{d}{d-\alpha}}(\Bn_r(x_0))} \le c\cEns_{\Bn_{r+\rho}(x_0)}(v,v) + c \rho^{-\alpha}\Vert v^2 \Vert_{L^1(\Bn_{r+\rho}(x_0))}.
\end{align}
\item[(ii)] Typically, if $\alpha \in (0,2)$, the Markov chains $(\Xn)$ converge to pure jump processes in $\R^d$. Therefore, the information on jumps of $(\Xn)$ to neighboring points in $\nZd$ do not survive in the limit $n \nearrow \infty$. Consequently, it is natural to impose only minimal assumptions on short connections. Allowing for $\sigma > 1$ in \eqref{Poinc} generalizes the class of admissible long-range conductances in the sense that it allows for conductances $(\Cn)_n$ that satisfy $\Cn(x,y) \equiv 0$ whenever $\vert x-y \vert \le \frac{\sigma}{n}$. Obviously, \eqref{Poinc}, \eqref{Sob-loc} fail for $\frac{1}{n} < r < \frac{\sigma}{2n}$ since $\cEn_{\Bn_r}(v,v) = 0$ for every $v \in L^2(\Bn_r)$.
\item[(iii)] \eqref{Poinc}, \eqref{Sob-loc} are trivially satisfied whenever $r,r+\rho < \frac{1}{n}$, regardless of $\Cn$.
\end{itemize}
\end{remark}

Clearly, assumption \eqref{CTail} implies \eqref{eq:ass0s}. Therefore, \autoref{lemma:markov-welldef} yields that $(\cEn, L^2(\nZd))$ is a regular lower bounded semi-Dirichlet form. We prove that under \eqref{K1} and \eqref{Sob}, we have that the constants in G\r{a}rding's inequality and the sector condition are independent of $n$.

\begin{lemma}
\label{lemma:elementary}
Assume that \eqref{K1}, \eqref{CTail} hold true for some $\alpha \in (0,2]$ and $\theta \in (\frac{d}{\alpha},\infty]$. Moreover, assume \eqref{Sob} if $\theta < \infty$.
Then, there are $c_1,c_2 > 0$ such that the following estimates hold true for every $u,v \in L^2(\nZd)$ and $n \in \N$:
\begin{align}
\label{eq:Garding-markov}
\cEn(u,u) &\ge \frac{1}{2} \cEns(u,u) -c_1 \Vert u \Vert^2_{L^2(\nZd)},\\
\label{eq:sector-markov}
\cEn(u,v)^2 &\le c_2 \cEns(u,u)\left(\cEns(v,v) + \Vert v \Vert^2_{L^2(\nZd)}\right).
\end{align}
\end{lemma}

\begin{proof}
First, as an easy consequence of H\"older's inequality, we obtain
\begin{align}
\label{eq:symsector}
\cEns(u,v)^2 \le \cEns(u,u)\cEns(u,u).
\end{align}
To treat the antisymmetric part, let us first denote $W(x) = n^{\alpha}\sum_{y \in \nZd} \frac{|\Cna(x,y)|^2}{J^{(n)}(x,y)}$ and prove the following auxiliary estimate (see Lemma 2.4 in \cite{KaWe22}) for $v \in L^2(\nZd)$ and $\delta > 0$:
\begin{align}
\label{eq:quantifiedK1consequence}
n^{-d} \sum_{x \in \nZd} v^2(x) W(x) \le \delta \cEns(v,v) + c_0 \delta^{\frac{d}{d-\theta\alpha}} \Vert W \Vert_{L^{\theta}(\nZd)}^{\frac{\theta \alpha}{\theta\alpha-d}} \Vert v\Vert_{L^2(\nZd)}^2,
\end{align}
where $c_0 > 0$ is some constant and we write $\frac{d}{d-\infty \alpha} = 0$, $ \frac{\infty \alpha}{\infty \alpha -d} = 1$. Note that in case $\theta = \infty$, \eqref{eq:quantifiedK1consequence} is a direct consequence of H\"older's inequality. If $\theta < \infty$, we decompose $W(x) = W_1(x) + W_2(x)$, where $W_1(x) = W(x) \mathbbm{1}_{\{|W(x)| > M\}}$ for some $M > 0$. We compute 
\begin{align*}
\Vert W_1 \Vert_{L^{\frac{d}{\alpha}}(\nZd)} &\le 2 \Vert W \Vert_{L^{\theta}(\nZd)} |\{ W \ge M \}|^{\frac{\alpha}{d} - \frac{1}{\theta}} \le 2 \Vert W \Vert_{L^{\theta}(\nZd)} \left( \frac{\Vert W \Vert_{L^{\theta}(\nZd)}}{M} \right)^{\theta \left( \frac{\alpha}{d} - \frac{1}{\theta} \right)}\\
&= 2 \Vert W \Vert_{L^{\theta}(\nZd)}^{\frac{\theta \alpha}{d}} M^{1 - \frac{\theta \alpha}{d}}.
\end{align*}
Now, let us choose $M = \left( \frac{\delta}{2} \right)^{\frac{d}{d-\theta\alpha}} \Vert W \Vert_{L^{\theta}(\nZd)}^{\frac{\theta \alpha}{\theta\alpha -d}}$. Then, $\Vert W_1 \Vert_{L^{\frac{d}{\alpha}}(\nZd)} < \frac{\delta}{c}$ and therefore
\begin{align*}
n^{-d} \sum_{x \in \nZd} v^2(x) W(x) \le \frac{\delta}{c} \Vert v^2 \Vert_{L^{\frac{d}{d-\alpha}}(\nZd)} + c_0 \delta^{\frac{d}{d-\theta\alpha}} \Vert W \Vert_{L^{\theta}(\nZd)}^{\frac{\theta \alpha}{\theta\alpha-d}} \Vert v\Vert_{L^2(\nZd)}^2,
\end{align*}
which yields \eqref{eq:quantifiedK1consequence} after application of \eqref{Sob}.\\
Having \eqref{eq:quantifiedK1consequence} at hand, we estimate
\begin{align}
\label{eq:nonsymsector}
\begin{split}
|\cEna(u,v)|^2 &\le \cE^{(n),J^{(n)}}(u,u) \left( n^{-d} \sum_{x \in \nZd} v^2(x) W(x) \right)\\
&\le c_1 \cEns(u,u) \left(\delta \cEns(v,v) + c_0 \delta^{\frac{d}{d-\theta\alpha}} \Vert W \Vert_{L^{\theta}(\nZd)}^{\frac{\theta \alpha}{\theta\alpha-d}} \Vert v\Vert_{L^2(\nZd)}^2 \right).
\end{split}
\end{align}
By combination of \eqref{eq:symsector} and \eqref{eq:nonsymsector}, we immediately obtain \eqref{eq:sector-markov}. To prove \eqref{eq:Garding-markov}, let us choose $\delta > 0$ so small that $c_1\delta < \frac{1}{2}$. Then, by application of \eqref{eq:nonsymsector} with $u=v$, we get
\begin{align*}
\cEn(u,u) \ge \cEns(u,u) - |\cEna(u,u)| \ge \frac{1}{2} \cEns(u,u) - c_2 \Vert u \Vert_{L^2(\nZd)}^2
\end{align*}
for some $c_2 > 0$, as desired.
\end{proof}

\begin{remark}
Note that in case $\theta = \frac{d}{\alpha}$, it is in general not possible to get $c_1, c_2 > 0$ independent of $n$ in \eqref{eq:Garding-markov}, \eqref{eq:sector-markov}. However, it is possible to prove \autoref{lemma:elementary} for $\theta = \frac{d}{\alpha}$ if $A$ is small enough.
\end{remark}

\subsection{Probability and PDEs}

Although our main results (see \autoref{thm:CLTl}, \autoref{thm:CLTnl}, \autoref{thm:approxl} and \autoref{thm:approxnl}) are of probabilistic nature, our analysis is based on the study of solutions $u$ to the heat equation, as well as the corresponding stationary equation
\begin{align*}
\partial_t u - L^{(n)} u = f, \qquad -L^{(n)} u = f
\end{align*}
associated with $L^{(n)}$. This section is meant to set up the weak solution
concept in the discrete setting and to introduce the heat semigroups and resolvents  associated with $(\cEn , L^2(\nZd))$ in order to prepare the proofs of our main results.

Given a connected set $\Bn \subset \nZd$, we introduce the function space 
\begin{align*}
L^2_c(\Bn) = \lbrace f \in L^2(\nZd) : \supp(f) \subset \Bn \rbrace,
\end{align*}
which will serve as a test function space for our solution concept. Solutions will all be contained in the following space
\begin{align*}
V(\Bn | \nZd) = \{ f : \nZd \to \R :& ~f\mid_{\Bn} \in L^2(\Bn),\\
& ~(f(x)-f(y))\vert\Cns(x,y)\vert^{1/2} \in L^2(\Bn \times \nZd) \}.
\end{align*}
Note that in particular $L^2(\nZd) \subset V(\Bn|\nZd)$ due to \autoref{lemma:markov-welldef}.

\begin{definition}
\label{def:weaksol}
Let $f \in L^{\infty}(\nZd)$, $\Bn \subset \nZd$ be connected and $I \subset \R$ be an interval.
\begin{itemize}
\item[(i)]  Let $\lambda \in \R$. We say that $u \in V(\Bn|\nZd)$ is a supersolution to $-L^{(n)} u + \lambda u = f$ in $\Bn$ if 
\begin{align}
\label{eq:superharm}
\cEn(u,\phi) + \lambda<u,\phi> \le <f,\phi> \text{ for all } \phi \in L^2_c(\Bn)\text{ with } \phi \le 0.
\end{align} 
$u$ is called a subsolution if \eqref{eq:superharm} holds true for every $\phi \ge 0$. $u$ is called solution, if it is a supersolution and a subsolution to $-L^{(n)} = f$.

\item[(ii)] We say that $u \in L^1_{loc}(I;L^2(\Bn))$ is a supersolution to $\partial_t u - L^{(n)} u = f$ in $I \times \Bn$ if the weak $L^2(\Bn)$-derivative $\partial_t u$ exists, $\partial_t u \in L^1_{loc}(I;L^2(\Bn))$ and 
\begin{align}
\label{eq:supercal}
n^{-d}\sum_{x \in \Bn} \partial_t u(t,x)\phi(x) + \cEn(u(t),\phi) \le <f,\phi>, ~~ \forall t \in I,~ \forall \phi \in L^2_c(\Bn) \text{ with } \phi \le 0.
\end{align}
$u$ is called a subsolution if \eqref{eq:superharm} holds true for every $\phi \ge 0$. $u$ is called solution, if it is a supersolution and a subsolution.
\end{itemize}
\end{definition}

In this article, we will mostly be concerned with solutions that are derived from the semigroup and resolvent corresponding to $\cEn$. For symmetric Dirichlet forms the interplay between the bilinear form, the associated semigroup and its Hunt process is a powerful tool and well-established in the literature. Although most connections remain valid in the nonsymmetric case, other properties fail in our situation. We provide a list of the features we will rely on in the sequel. All results are standard and can be found in \cite{FuUe12}, \cite{Osh13}, \cite{DeKu13}, or \cite{MSS18}.

\subsubsection{Semigroups and heat kernels}

The heat kernel for $\Xn$, defined by 
\begin{align*}
p^{(n)}_t(x,y) = n^d \PP^x(\Xn_t = y)
\end{align*}
induces the transition semigroup $(P^{(n)}_t)_{t > 0}$ given as
\begin{align*}
P^{(n)}_t f(x) = n^{-d} \sum_{y \in \nZd} p^{(n)}_t(x,y) f(y) = \EE^x(f(\Xn_t)), ~~ f \in L^2(\nZd).
\end{align*}
$(P^{(n)}_t)$ coincides with the strongly continuous semigroup that is associated to $\cEn$ via the theory of lower bounded semi-Dirichlet forms and is strongly continuous and Markovian (see Chapter 3 in \cite{Osh13}). Therefore we denote both objects by $(P^{(n)}_t)$.\\ Moreover, it holds that for every $f \in L^2(\nZd)$, $(t,x) \mapsto P^{(n)}_tf(x)$ is a solution to $\partial_t u - L^{(n)} u = 0$ in $(0,\infty) \times \nZd$ satisfying $ \Vert P^{(n)}_tf - f \Vert_{L^2(\nZd)} \to 0$, as $t \searrow 0$.

Moreover, for any set $\Bn \subset \nZd$ we introduce the restricted form $(\cEn,L_c^2(\Bn))$. Then $(\cEn,L_c^2(\Bn))$ is a regular lower-bounded semi-Dirichlet form on $L^2(\Bn)$ with heat semigroup $(P_t^{\Bn})$ on $L^2(\Bn)$ defined by
\begin{align*}
P_t^{\Bn} f(x) = \EE^x(\mathbbm{1}_{\lbrace t \le \tau_{\Bn} \rbrace} f(\Xn_t)), ~~ f \in L^2(\Bn),
\end{align*}
where $\tau_{\Bn} = \inf \lbrace t > 0 : \Xn_t \not\in \Bn\rbrace$ is the first exit time of $\Bn$.\\
By the definition of $(\cEn,L_c^2(\Bn))$ it becomes apparent that $((t,x) \mapsto P^{\Bn}_t f(x)) \in L^2_c(\Bn)$ is a solution to $\partial_t u  - L^{(n)} u = 0$ in $(0,\infty) \times \Bn$ with initial data $f \in L^2(\Bn)$. \\
The process associated with the restricted form is the killed process $X^{\Bn}$ given by 
\begin{align*}
X^{\Bn}_t = \begin{cases}
\Xn_t,& 0 \le t < \tau_{\Bn},\\
\partial,& t \ge \tau_{\Bn},
\end{cases}
\end{align*}
where $\partial$ denotes the cemetery state. We refer to Chapter 3.5 \cite{Osh13} for the exact construction of the restricted form, the killed process and their correspondence.

\subsubsection{Resolvent operators}

From G\r{a}rding's inequality \eqref{eq:Garding-markov}, we know that $\cEn(u,u) + \lambda(u,u) \ge 0$ if $\lambda \ge c_1 =: \lambda_0$. For any $\lambda > \lambda_0$ one can define the resolvent operator
\begin{align}
\label{eq:defresolvent}
U^{(n)}_{\lambda} f(x) = \int_{0}^{\infty} e^{-\lambda t} P_t^{(n)} f(x) \d t, ~~ f \in L^2(\nZd).
\end{align}
It holds 
\begin{align}
\label{eq:resolventsol}
\cEn(U^{(n)}_{\lambda} f,g) + \lambda(U^{(n)}_{\lambda} f,g) = (f,g), \quad \forall f,g \in L^2(\nZd),
\end{align}
which is why $x \mapsto U^{(n)}_{\lambda} f(x)$ solves $-L^{(n)} u + \lambda u = f$ in $\nZd$. $U^{(n)}_{\lambda} f$ is the unique element in $L^2(\nZd)$ with this property (see Theorem 1.1.2 in \cite{Osh13}). Moreover, by continuity of $U_{\lambda}^{(n)}$ and \eqref{eq:resolventsol} there exists $c(\lambda,\lambda_0) > 0$ such that
\begin{align}
\label{eq:resolventestimate}
\cEn(U_{\lambda}^{(n)}f, U_{\lambda}^{(n)}f) + \Vert U_{\lambda}^{(n)}f \Vert_{L^2(\nZd)}^2 \le c(\lambda,\lambda_0) \Vert f \Vert_{L^2(\nZd)}^2.
\end{align}

\section{Regularity properties of solutions}
\label{sec:sol}

In order to establish convergence of the Markov chains $(X^{(n)})$, we require several qualitative properties of solutions to the parabolic equation $\partial_t u - L^{(n)} u = 0$, respectively the elliptic equation $- L^{(n)}u = 0$. Weak Harnack inequalities and interior H\"older estimates can be deduced via the methods applied in \cite{KaWe22}. Moreover, we prove a weak maximum principle for subsolutions to the parabolic equation.

\subsection{Regularity and weak Harnack inequality}

In this section we establish a weak Harnack inequality for supersolutions to $\partial_t u - L^{(n)} u = 0$ and $-L^{(n)} u = 0$, as well as interior H\"older estimates for solutions to these equations. It is a crucial feature of these results that the constants do not depend on $n$, the solution $u$, or the diameter of the solution domain, but only on the family of conductances $(C_n)$ itself, through the constants in the underlying assumptions \eqref{K1}, \eqref{K2}, \eqref{Poinc}, \eqref{Sob}, \eqref{CTail}, \eqref{CTail2}. This uniformity in $n$ renders possible convergence of the laws of the corresponding Markov chains in our main results. While for symmetric Markov chains on $\nZd$ such phenomenon occurs naturally by appropriate scaling of a given chain on $\Z^d$, we have to explicitly prescribe the correct behavior for $n \nearrow \infty$ due to the lack of symmetry in our setup (see \eqref{K1}).

We now state all regularity results that will be needed in the subsequent chapters. It is important to point out that these results were already established by the author in \cite{KaWe22} for integro-differential operators on $\R^d$ governed by nonsymmetric jumping kernels under similar assumptions. The proofs of the corresponding results in Euclidean space do not differ from the discrete setting. Therefore, we only present a sketch of the proofs in this article and refer to \cite{KaWe22} for a detailed discussion.

\begin{theorem}[weak parabolic Harnack inequality]
\label{thm:PHI}
Assume that \eqref{K1}, \eqref{K2}, \eqref{Poinc}, \eqref{Sob}, \eqref{CTail} hold true for some $\alpha \in (0,2]$, $\sigma > 0$ and $\theta \in (\frac{d}{\alpha},\infty]$. Then there exists $c > 0$ such that for every $n \in \N$, $\frac{\sigma}{n} < R \le 1$, $x_0 \in \nZd$, and every nonnegative supersolution $u$ to $\partial_t u -L^{(n)}u = 0$ in $(t_0 - R^{\alpha}, t_0 + R^{\alpha}) \times \Bn_{2R}(x_0)$:
\begin{align}
\label{eq:PHI}
\dashint_{(t_0 - R^{\alpha}, t_0 - R^{\alpha} + (\frac{R}{2})^{\alpha})} \left(\left(\frac{nR}{2} \right)^{-d}\sum_{x \in \Bn_{\frac{R}{2}}(x_0)} u(t,x)\right) \d t \le c \inf_{(t_0+R^{\alpha} - (\frac{R}{2})^{\alpha}, t_0 + R^{\alpha}) \times \Bn_{\frac{R}{2}}(x_0)} u.
\end{align}
\end{theorem}

\begin{remark}
In particular, under the assumption of \autoref{thm:PHI}, there exists $c > 0$ such that for every $n \in \N$, $\frac{\sigma}{n} < R \le 1$, $x_0 \in \nZd$, and every nonnegative supersolution $u$ to $-L^{(n)}u = 0$ in $\Bn_{2R}(x_0)$, it holds
\begin{align}
\label{eq:EHI}
\left(\frac{nR}{2} \right)^{-d}\sum_{x \in \Bn_{\frac{R}{2}}(x_0)} u(t,x)\le c \inf_{\Bn_{\frac{R}{2}}(x_0)} u.
\end{align}
\end{remark}

\begin{theorem}[parabolic H\"older estimates]
\label{thm:PHR}
Assume that \eqref{K1}, \eqref{K2}, \eqref{Poinc}, \eqref{Sob}, \eqref{CTail},  \eqref{CTail2} hold true for some $\alpha \in (0,2]$, $\sigma > 0$ and $\theta \in (\frac{d}{\alpha},\infty]$. Then there exist $c > 0$ and $\gamma \in (0,1)$ such that for every $n \in \N$, $\frac{\sigma}{n} < R \le 1$, $x_0 \in \nZd$ and every solution $u$ to $\partial_t u -L^{(n)}u = 0$ in $(t_0 - R^{\alpha}, t_0 + R^{\alpha}) \times \Bn_{2R}(x_0)$, it holds
\begin{align}
\label{eq:PHR}
\vert u(t,x) - u(s,y)\vert \le c\Vert u\Vert_{L^{\infty}([t - R^{\alpha}, t + R^{\alpha}] \times \nZd)} \left(\frac{\vert t-s\vert^{1/\alpha} + \vert x-y \vert}{R}\right)^{\gamma}
\end{align}
for a.e. $(t,x),(s,y) \in (t_0-R^{\alpha},t_0+R^{\alpha}) \times \Bn_R(x_0)$ with $x \neq y$. Moreover:
\begin{align}
\label{eq:PHR2}
\vert u(t,x) - u(s,x)\vert \le c\Vert u\Vert_{L^{\infty}([t - R^{\alpha}, t + R^{\alpha}] \times \nZd)} \left(\frac{\vert t-s\vert^{1/\alpha} \vee \frac{\sigma}{n}}{R}\right)^{\gamma}
\end{align}
for a.e. $t,s \in (t_0-R^{\alpha},t_0+R^{\alpha})$, $x \in \Bn_R(x_0)$.
\end{theorem}

\begin{remark}
In particular, under the assumption of \autoref{thm:PHR}, there exist $c > 0$ and $\gamma \in (0,1)$ such that for every $n \in \N$, $\frac{\sigma}{n} < R \le 1$, $x_0 \in \nZd$, and every solution $u$ to $-L^{(n)}u = 0$ in $\Bn_{2R}(x_0)$, it holds
\begin{align}
\label{eq:EHR}
\vert u(x) - u(y)\vert \le c \Vert u\Vert_{L^{\infty}(\nZd)}\left(\frac{\vert x-y \vert}{R}\right)^{\gamma}
\end{align}
for a.e. $x,y \in \Bn_R(x_0)$.
\end{remark}

\begin{proof}[Proof of \autoref{thm:PHI}]
It is straightforward to adapt all proofs in \cite{KaWe22} line by line to the discrete setup at hand.
First, we define $\widetilde{C}^{(n)}(x,y) := n^{\alpha+d}\Cn(x,y)$ and $\int_{\Bn} f(x) \d x := n^{-d}\sum_{x \in \Bn} f(x)$ for $f : \Bn \to \R$, where $\Bn \subset \nZd$. Then, we can rewrite
\begin{align}
\label{eq:bilform}
\cEn(u,v) = \int_{\nZd} \int_{\nZd} (u(x)-u(y))v(x) \widetilde{C}^{(n)}(x,y) \d y \d x.
\end{align}  
This resembles the exact shape of the forms considered in \cite{KaWe22}. 
Upon introducing the notation $\widetilde{J}^{(n)}(x,y) = n^{\alpha+d}J^{(n)}(x,y)$, $\widetilde{j}^{(n)}(x,y) = n^{\alpha+d}j^{(n)}(x,y)$ for the auxiliary jumping kernels $J^{(n)}$ and $j^{(n)}$ from assumptions \eqref{K1} and \eqref{K2}, we get the following estimates from the fact that the assumptions \eqref{K1}, \eqref{K2}, \eqref{Poinc}, \eqref{Sob}, \eqref{CTail}, and \eqref{CTail2} are assumed to hold true:
\begin{align*}
\left\Vert  \int_{\nZd} \frac{\vert \widetilde{C_a}^{(n)}(\cdot,y)\vert^{2}}{\widetilde{J}^{(n)}(\cdot,y)} \d y \right\Vert_{L^{\theta}(\nZd)} &\le A,\\
(1-D) \widetilde{j}^{(n)}(x,y) &\le \widetilde{C}^{(n)}(x,y), ~~ x,y \in \Bn_2,\\
\sup_{x \in \nZd} \int_{\nZd \setminus \Bn_r(x)} \widetilde{C_s}^{(n)}(x,y) &\le c r^{-\alpha}, ~~ 0 < r \le 1,\\
\sup_{x \in \nZd} \int_{\nZd \setminus \Bn_r(x)} \widetilde{C_s}^{(n)}(x,y) &\le c r^{-\delta}, ~~ 1 \le r < \infty,\\
\int_{\Bn_r} (v(x) - [v]_{\Bn_r})^2 \d x &\le c r^{\alpha} \cEns_{\Bn_r}(v,v), ~~ \frac{\sigma}{2n} < r \le 1,\\
\Vert v^2 \Vert_{L^{\frac{d}{d-\alpha}}(\nZd)} &\le c \cEns(v,v).
\end{align*} 
These are the exact analogs of the corresponding conditions (K1), (K2), (Cutoff), (Poinc), (Sob), ($\infty$-Tail) from \cite{KaWe22}. In particular, the constants $A,D,c > 0$ do not depend on $n$. This allows us to follow the proof in \cite{KaWe22} line by line.\\
Let us mention that the restriction to radii $R > \frac{\sigma}{n}$ in \autoref{thm:PHI} is due to the fact that in this regime $\nZd$ satisfies the volume regularity property \eqref{eq:volreg}. This is crucial for the proof of  \autoref{thm:PHI}. Moreover, by carefully tracking the proofs in \cite{KaWe22}, it becomes apparent that in order to show weak Harnack inequalities for solutions on cylinders $(t_0-R^{\alpha},t_0+R^{\alpha}) \times \Bn_{2R}$, for fixed $R > \frac{\sigma}{n}$, it suffices to have \eqref{Poinc} for $r \in (R,2R)$ and \eqref{Sob} for $r \in (R/2,R)$. We refer the interested reader to \cite{Bos20}, from where the exact dependencies can be read off in the symmetric case.
\end{proof}

\begin{proof}[Proof of \autoref{thm:PHR}]
For the proof of H\"older estimates, one finds that Harnack inequalities for solutions on cylinders as above, where $R > \frac{\sigma}{n}$, yield the correct estimate \eqref{eq:PHR} for any $(t,x),(s,y) \in (t_0-R^{\alpha},t_0+R^{\alpha}) \times \Bn_R$ with $\vert t-s \vert^{1/\alpha} + \vert x-y \vert \ge \frac{\sigma}{n}$ (see \cite{KaWe22}). For $(t,x),(s,y) \in (t_0-R^{\alpha},t_0+R^{\alpha}) \times \Bn_R$ with $\vert t-s \vert^{1/\alpha} + \vert x-y \vert \le \frac{\sigma}{n}$, the same proof yields
\begin{align}
\label{eq:HRhelp}
\vert u(t,x) - u(s,y)\vert \le c\Vert u\Vert_{L^{\infty}([t - R^{\alpha}, t + R^{\alpha}] \times \nZd)} \left(\frac{\frac{\sigma}{n}}{R}\right)^{\gamma}.
\end{align}
However, in this case case, there are only two possibilities. Either, we have $x = y$, in which case we directly obtain the desired estimate \eqref{eq:PHR2} from \eqref{eq:HRhelp}. Alternatively, it holds $x \neq y$, which implies
that $|x-y| \ge \frac{1}{n}$. This already the estimate \eqref{eq:PHR}.
\end{proof}

\begin{remark}
It is important to point out that the case $\alpha = 2$ does not differ from the case $\alpha \in (0,2)$. This is due to the fact that all proofs in \cite{KaWe22} are robust in the sense that they work for any bilinear form of type \eqref{eq:bilform} governed by an integral kernel as long as assumptions \eqref{K1}, \eqref{K2}, \eqref{Poinc}, \eqref{Sob}, \eqref{CTail}, \eqref{CTail2} are satisfied for some $\alpha \in (0,2]$. All constants depend on $\alpha$ only through the constants in the assumptions. However, in comparison to the continuous case, assumption \eqref{eq:ass0} guarantees that $\cEn(u,u)$ is well-defined for any $u \in L^2(\nZd)$, also when $\alpha = 2$, and not only for $u \equiv 0$.
\end{remark}

Next, we present a corollary from \autoref{thm:PHR} which yields H\"older estimates for the resolvent $U^{(n)}_{\lambda} f$ and the heat semigroup $P_t^{(n)} f$, for $f \in L^{\infty}(\nZd)$. Let $\lambda_0 > 0$ be as in \autoref{lemma:elementary}. These estimates will become crucial in \autoref{sec:convergence}.

\begin{corollary}
\label{cor:regest}
Assume that \eqref{K1}, \eqref{K2}, \eqref{Poinc}, \eqref{Sob}, \eqref{CTail}, \eqref{CTail2} hold true for some $\alpha \in (0,2]$, $\sigma > 0$ and $\theta \in (\frac{d}{\alpha},\infty]$. Then there exist $c > 0$ and $\gamma \in (0,1)$ such that for every $\lambda > \lambda_0$, $t,s > 0$, $f \in L^{\infty}(\nZd) \cap L^2(\nZd)$ and every $n \in \N$ and $x,y \in \nZd$:
\begin{align}
\label{eq:sgHRE}
\vert P^{(n)}_tf(x) - P^{(n)}_tf(y)\vert &\le c \Vert f \Vert_{L^{\infty}(\nZd)} \vert x-y \vert^{\gamma},\\
\label{eq:sgHRE2}
\vert P^{(n)}_tf(x) - P^{(n)}_sf(x)\vert &\le c \Vert f \Vert_{L^{\infty}(\nZd)} \max\left(\vert t-s\vert^{1/\alpha} , \frac{\sigma}{n}\right)^{\gamma},\\
\label{eq:resHRE}
\vert U^{(n)}_{\lambda} f(x) - U^{(n)}_{\lambda} f(y)\vert &\le c \lambda^{-1}\Vert f \Vert_{L^{\infty}(\nZd)} \vert x-y \vert^{\gamma},
\end{align}
\end{corollary}

\begin{proof}
First, we observe that \eqref{eq:resHRE} is a direct consequence of \eqref{eq:sgHRE} and the observation that
\begin{align*}
\vert U^{(n)}_{\lambda} f(x) - U^{(n)}_{\lambda} f(y)\vert \le \int_0^{\infty} e^{-\lambda t} \vert P^{(n)}_tf(x) - P^{(n)}_tf(y)\vert \d t \le \frac{1}{\lambda} \sup_{t > 0}\vert P^{(n)}_tf(x) - P^{(n)}_tf(y)\vert.
\end{align*} 
For the proof of estimate \eqref{eq:sgHRE} we recall that $(t,x) \mapsto P^{(n)}_tf(x)$ solves $\partial_t u - L^{(n)}u = 0$ in $(0,\infty) \times \nZd$, which is why \autoref{thm:PHR} applied with $R = 1$ on arbitrary time-space cylinders in $(0,\infty) \times \nZd$ yields
\begin{align*}
\vert P^{(n)}_tf(x) - P^{(n)}_tf(y)\vert \le c \Vert P^{(n)}_t \vert f \vert\Vert_{L^{\infty}(\nZd)} \vert x-y \vert^{\gamma} \le c \Vert f \Vert_{L^{\infty}(\nZd)} \vert x-y \vert^{\gamma}
\end{align*}
for any $t > 0$ and $\vert x-y \vert \le \frac{1}{2}$. Note that we also used that $\vert P^{(n)}_t \vert f \vert \vert \le \Vert f \Vert_{L^{\infty}(\nZd)}$ which is due to Markovianity of $(P_t^{(n)})$. In case $\vert x-y \vert > \frac{1}{2}$, estimate \eqref{eq:sgHRE} is trivial since
\begin{align*}
\vert P^{(n)}_tf(x) - P^{(n)}_tf(y)\vert \le 2 \Vert P^{(n)}_t \vert f \vert \Vert_{L^{\infty}(\nZd)} \le 2^{1+\gamma} \Vert f \Vert_{L^{\infty}(\nZd)} \vert x-y \vert^{\gamma}.
\end{align*}
Estimate \eqref{eq:sgHRE2} can be proved from \eqref{eq:PHR2} similar to how \eqref{eq:sgHRE} is proved from \eqref{eq:PHR}.
\end{proof}

The following result is a discrete version of Lemma 3.1 from \cite{KaWe22} and can be proved in the exact same fashion. It is needed in the proof of \autoref{lemma:ETE}. 

\begin{lemma}[$\log(u)$-estimate]
\label{lemma:logu}
Assume that \eqref{K1}, \eqref{K2} and \eqref{CTail} hold true for some $\alpha \in (0,2]$ and $\theta \in (\frac{d}{\alpha},\infty]$. Then there exist $c_1, c_2 > 0$ such that for every $\frac{\sigma}{2n} < r \le 1$, $0 < \rho \le r$ and every nonnegative function $u \in V(\Bn_{2r} | \nZd)$ that satisfies $u > \eps$ in $\Bn_{2r}$ for some $\eps > 0$:
\begin{align*}
c_1 n^{\alpha-d}\hspace{-0.2cm}\sum_{x \in \Bn_{r+\rho}} \sum_{y \in \Bn_{r+\rho}} \hspace{-0.2cm}\tau(x)\tau(y) \left(\log\frac{u(x)}{\tau(x)}-\log\frac{u(y)}{\tau(y)}\right)^2 \hspace{-0.2cm}\Cns(x,y) \le \cEn(u,-\tau^2 u^{-1}) + c_2 \rho^{-\alpha}\mu^{(n)}( \Bn_{r+\rho}),
\end{align*}
where (as in \eqref{eq:cutoff}) $\tau : \nZd \to [0,1]$ satisfies $\supp(\tau) = \Bn_{r+\rho}$, $\tau \equiv 1$ in $\Bn_r$ and\\ $\max_{i = 1,...,d} \Vert\nabla^{(n)}_i \tau\Vert_{L^{\infty}(\nZd)} \le 2\rho^{-1}$.
\end{lemma}

\begin{remark}[extensions and simplifications]
\begin{itemize}
\item[(i)] For nonnegative solutions $u$ to $\partial_t u -L^{(n)}u = 0$ in a suitable time-space cylinder one can also establish a full Harnack inequality under slightly stronger assumptions. We refer to \cite{KaWe22b} for a discussion of such estimate for nonsymmetric integro-differential operators in $\R^d$.

\item[(ii)] One can prove \eqref{eq:PHI}, \eqref{eq:EHI}, \eqref{eq:PHR}, \eqref{eq:EHR} for $L^{(n)}$ and \autoref{cor:regest}, \autoref{lemma:logu} also under the following localized condition instead of \eqref{K1}:\\ 
There exists $C > 0$ such that for every $n \in \N$, $x_0 \in \nZd$:
\begin{align}
\label{K1loc}\tag{$\text{K1}_{loc}$}
\left\Vert n^{\alpha}\sum_{y \in \Bn_2(x_0)} \frac{\vert\Cna(\cdot,y)\vert^2}{\Cns(\cdot,y)} \right\Vert_{L^{\theta}(\Bn_{2}(x_0))} \le C.
\end{align}

\item[(iii)] We emphasize that one can prove \eqref{eq:resHRE} directly via establishing a version of \eqref{eq:EHR} for solutions to $-L^{(n)} u + \lambda u = f$, where $f \in L^{\infty}(\nZd)$ and $\lambda > 0$.
The existence of a killing term $\lambda u$ for $\lambda \ge 0$ in \eqref{eq:EHR} does not influence the proofs significantly since the term $\lambda u$ can be treated in a similar way as a source term.
\end{itemize}
\end{remark}

\subsection{Weak parabolic maximum principle}

In this section we provide a parabolic maximum principle for weak subsolutions to $\partial_t u - L^{(n)} u = 0$ in the spirit of Proposition 5.2 in \cite{GHL09}, where such result was proved in the symmetric case. An elliptic version for nonsymmetric operators has been established in Theorem 4.1 in \cite{FKV15}. Let us point out that all constants in the proof below might (and are allowed to) depend on $n \in \N$.

\begin{proposition}
\label{prop:pmp}
Assume that \eqref{K1}, \eqref{Sob} hold true for some $\alpha \in (0,2]$, $\sigma > 0$ and $\theta \in (\frac{d}{\alpha},\infty]$. Let $n \in \N$, $T > 0$ and $\frac{\sigma}{n} < R \le 1$. Let $u$ be a subsolution to $\partial_t u - L^{(n)} u = 0$ in $(0,T) \times \Bn_{R}$ for some ball $\Bn_R \subset\nZd$, such that
\begin{itemize}
\item $u_+(t) \in L^2_c(\Bn_R)$ for every $t \in (0,T)$,
\item $u_+(t) \to 0$ in $L^2(\Bn_{R})$ as $t \searrow 0$ and $u(0) \le 0$ in $\Bn_R$.
\end{itemize}
Then $u \le 0$ a.e. in $(0,T) \times \Bn_R$.
\end{proposition}

\begin{proof}
Note that
\begin{equation}
\label{eq:pmp1help0}
(u(x) - u(y))u_+(x) = (u_+(x)-u_+(y))u_+(x) + u_-(y)u_+(x).
\end{equation}
By G\r{a}rding's inequality \eqref{eq:Garding-markov} it follows that for every $t > 0$
\begin{equation}
\label{eq:pmp1help1}
\cEn(u(t),u_+(t)) \ge \cEn(u_+(t),u_+(t)) \ge \frac{1}{2} \cEns(u_+(t),u_+(t)) - c_1 \Vert u_+(t) \Vert_{L^2(\nZd)}^2
\end{equation}
for some $c_1 > 0$. 
Since $u_+(t) \in L^2_c(\Bn_R)$, we can test the weak formulation of $\partial_t u - L^{(n)} u = 0$ with $u_+(t)$ for every $t \in (0,T)$. Then, we integrate in time over an arbitrary interval $(t_1,t_2) \subset (0,T)$, apply integration by parts formula and obtain that for every $(t_1,t_2) \subset (0,T)$:
\begin{equation}
\label{eq:steklovapplmax}
n^{-d} \sum_{x \in \Bn_{R}} u_+^2(t_2,x) - n^{-d}  \sum_{x \in \Bn_{R}} u_+^2(t_1,x) + \int_{t_1}^{t_2} \cEn(u(t),u_+(t)) \d t \le 0.
\end{equation}
Note that the above explanation can be made rigorous with the help of Steklov averages (see \cite{FeKa13}, or \cite{KaWe22b}). From \eqref{eq:pmp1help1} and \eqref{eq:steklovapplmax}, it follows:
\begin{equation*}
n^{-d} \hspace{-0.2cm} \sum_{x \in \Bn_{R}}\hspace{-0.2cm} u_+^2(t_2,x) - n^{-d} \hspace{-0.2cm} \sum_{x \in \Bn_{R}}\hspace{-0.2cm} u_+^2(t_1,x) + \frac{1}{2}\int_{t_1}^{t_2} \hspace{-0.2cm} \cEns(u_+(t),u_+(t)) \d t \le c_1 n^{-d} \hspace{-0.2cm} \int_{t_1}^{t_2} \hspace{-0.2cm} \sum_{x \in \Bn_{R}} \hspace{-0.2cm} u_+^2(t,x)  \d t.
\end{equation*}

Let us define $A(t) := n^{-d} \sum_{x \in \Bn_{R}} u_+^2(t,x)$ and set $\delta := \frac{1}{2c_1}$. Then, for every $t_1 \in (0,T)$:
\begin{equation*}
\sup_{t \in (t_1,t_1 + \delta)} A(t) + \frac{1}{2} \int_{t_1}^{t_1 + \delta} \cEns(u_+(t),u_+(t)) \d t \le \frac{1}{2}\sup_{t \in (t_1,t_1 + \delta)} A(t) + A(t_1),
\end{equation*}
so for every $k > 0$:
\begin{equation}
\sup_{t \in (t_1,t_1 + k\delta)} A(t) \le 2^k A(t_1).
\end{equation}
Since $A(t) \searrow 0$, as $t \searrow 0$, by assumption, it follows that $A(t) \le 0$ for every $t \in (0,T)$, which gives the desired result.
\end{proof}

\begin{remark}
\begin{itemize}
\item[(i)] As a standard application of \autoref{prop:pmp} we have the following monotonicity property for the restricted semigroups: Whenever $\Bn_1 \subset \Bn_2$ for two sets $\Bn_1, \Bn_2 \subset \nZd$, it holds for every nonnegative $f \in L^{\infty}(\nZd)$:
\begin{align}
\label{eq:ressgmon}
P^{\Bn_1}_t f(x) \le P^{\Bn_2}_t f(x), ~~\forall t > 0,~ x \in \nZd.
\end{align}
A proof can be found in \cite{GrHu08} (Lemma 4.16). Note that \eqref{eq:ressgmon} is immediate, when using the representation via the corresponding Markov chains $X^{(n)}$.
\item[(ii)] In particular it holds that $P^{\Bn_1}_t f(x) \le P^{(n)}_t f(x)$, which follows by iteratively applying \eqref{eq:ressgmon} to a sequence of balls $(\Bn_{R_i})_{i}$ with $R_i \nearrow \infty$. 
\end{itemize}
\end{remark}

\section{Tightness}
\label{sec:tightness}

The goal of this section is to establish the following theorem:

\begin{theorem}
\label{thm:tightnessprep}
Let $A > 0$, $B \in (0,1)$. Assume that \eqref{K1}, \eqref{K2}, \eqref{Poinc}, \eqref{Sob}, \eqref{CTail} hold true for some $\alpha \in (0,2]$, $\sigma > 0$ and $\theta \in (\frac{d}{\alpha},\infty]$. Then there exists $t_0 \in (0,1)$ such that for every $0 < R \le 1$ there is $N \in \N$ such that for every  $n \ge N$ and $x \in \nZd$ it holds:
\begin{align*}
\PP^x\left(\sup_{t \le (t_0 R)^{\alpha}} \vert \Xn_t - x \vert > AR \right) \le B.
\end{align*}
\end{theorem}

This result implies tightness of the laws of $(X^{(n)})_n$ in the Skorohod space $D([0,T];\R^d)$ for every $T > 0$ (see \autoref{thm:tightness}).
Such statement is standard in the literature for symmetric Markov chains (see Proposition 3.1 in \cite{BKU10}, Proposition 3.4 in \cite{BaKu08}, or Theorem 3.1 in \cite{HuKa07}) and was also proved in \cite{DeKu13} (see Proposition 3.3). The restriction to $0 < R \le 1$ and $n \ge N$ stems from the admissible range of radii in \autoref{thm:PHI}.

We prove \autoref{thm:tightnessprep} by adapting the arguments from Chapter 7 in \cite{Bos20} to a nonsymmetric setting. The procedure loosely follows the path laid out by Grigor'yan and coauthors (see \cite{GHL14}, \cite{GHH17}, \cite{GHH18}) who investigate the validity of heat kernel bounds for symmetric regular Dirichlet forms on metric measure spaces in connection to geometric properties of the underlying space. It turns out that the weak Harnack inequality (see \autoref{thm:PHI}) implies two-sided estimates for exit times of $X^{(n)}_t$, which in turn imply a so-called survival estimate. The survival estimate has proved to be a helpful tool for the derivation of heat kernel bounds but can also be applied when proving \autoref{thm:tightnessprep}. Our proof uses an iteration technique reminiscent of \cite{GHL14} but avoids truncation of conductances (see also \cite{GrHu14}). We point out that the proofs of all results in this section are purely analytic, the main tool being the weak Harnack inequality \autoref{thm:PHI} and the parabolic maximum principle \autoref{prop:pmp}.

First, we establish an estimate for the exit times of $(X^{(n)})$ from balls $\Bn_r$. While the proof of the lower bound \eqref{eq:ETE2} works as for Lemma 5.5 in \cite{GHH17}, establishing the upper bound \eqref{eq:ETE1} is more involved due to the lack of symmetry. We give a proof, inspired by Proposition 3.1 in \cite{Del99}, that works for long- and bounded range at the same time and does not make use of the dual semigroup $\widehat{P}_t^{(n)}$.  

Let us introduce the Green operator $G^{\Bn}$ defined by
\begin{align*}
G^{\Bn} f(x) = \int_{0}^{\infty} P^{\Bn}_t f(x) \d t = \EE^x \left( \int_{0}^{\tau_{\Bn}} f(\Xn_t) \d t \right), ~~f\in L^2(\Bn). 
\end{align*}
In particular, $G^{\Bn} \mathbbm{1}(x) = \EE^x \tau_{\Bn}$ for every $x \in \Bn$. 

\begin{lemma}
\label{lemma:ETE}
Assume that \eqref{K1}, \eqref{K2}, \eqref{Poinc}, \eqref{Sob}, \eqref{CTail} hold true for some $\alpha \in (0,2]$, $\sigma > 0$ and $\theta \in (\frac{d}{\alpha},\infty]$. Then there exist $c_1,c_2 > 0$ such that for every $n \in \N$, $\frac{32\sigma}{n} < R \le 1$, $x_0 \in \nZd$:
\begin{align}
\label{eq:ETE2}
\EE^x\left(\tau^{(n)}_{\Bn_{R/8}(x_0)}\right) &\le c_1 R^{\alpha} , ~~ \forall x \in \Bn_{R/8},\\
\label{eq:ETE1}
\EE^x\left(\tau^{(n)}_{\Bn_{R/8}(x_0)}\right) &\ge c_2 R^{\alpha}, ~~\forall x \in \Bn_{R/32}.
\end{align}
\end{lemma}

\begin{proof}
We start with the proof of the first inequality \eqref{eq:ETE2}. First, we prove the following lower bound on the heat semigroup $(P^{(n)}_t)$: There exist $c,\eps > 0$ such that for every $y_0 \in \nZd$ and $\frac{32\sigma}{n} < R \le 1$ it holds
\begin{align}
\label{eq:NDLHKB}
P^{(n)}_{t_0} \mathbbm{1}_{\Bn_{R/16}(y_0)}(y) \ge \eps, ~~\forall y \in \Bn_{R/2}(y_0),
\end{align}
where $t_0 := cR^{\alpha}$.\\
Let $y_0 \in \nZd$, $\frac{32\sigma}{n} < R \le 1$ be given. We define 
\begin{align*}
u(s,w) = \begin{cases}
&1, ~\text{ if } s \le (R/32)^{\alpha},\\
&P^{(n)}_{s-(R/32)^{\alpha}}\mathbbm{1}_{\Bn_{R/16}(y_0)}(w), ~\text{ if } s > (R/32)^{\alpha}.
\end{cases}
\end{align*}
Then $u$ is nonnegative and solves $\partial_t u - L^{(n)}u = 0$ in $(0,2(R/32)^{\alpha}) \times \Bn_{R/16}(y_0)$. In particular, note that $\partial_t u(w)$ exists for every $w \in \Bn_{R/16}(y_0)$ since $t \mapsto u(t,w)$ is absolutely continuous.
Using the weak parabolic Harnack inequality for $L^{(n)}$ (see \autoref{thm:PHI}), we compute:
\begin{align*}
1 &= \frac{n^{-d}}{\mu^{(n)}(\Bn_{R/64}(y_0))}\sum_{y \in \Bn_{R/64}(y_0)}\dashint_{(0,(R/64)^{\alpha})}\frac{n^{-d}}{\mu^{(n)}(\Bn_{R/64}(y_0))}\sum_{w \in \Bn_{R/64}(y_0)} u(s,w) \d s\\
&\le c_1 (nR/64)^{-d}\sum_{y \in \Bn_{R/64}(y_0)} \inf_{(s,w) \in (2(R/32)^{\alpha}-(R/64)^{\alpha},2(R/32)^{\alpha}) \times \Bn_{R/64}(y_0)} u(s,w)\\
&\le c_1 (nR/64)^{-d}\sum_{y \in \Bn_{R/64}(y_0)} \dashint_{(2(R/32)^{\alpha}-(R/64)^{\alpha},2(R/32)^{\alpha})} u(s,y) \d s\\
&= c_1 (nR/64)^{-d}\sum_{y \in \Bn_{R/64}(y_0)} \dashint_{((R/32)^{\alpha}-(R/64)^{\alpha},(R/32)^{\alpha})} P^{(n)}_{s}\mathbbm{1}_{\Bn_{R/16}(y_0)}(y) \d s\\
&= c_1 n^{-d}\sum_{z \in \Bn_{R/16}(y_0)} \dashint_{((R/32)^{\alpha}-(R/64)^{\alpha},(R/32)^{\alpha})} (nR/64)^{-d}\sum_{y \in \Bn_{R/64}(y_0)}  p^{(n)}_{s}(y,z) \d s,
\end{align*}
where $c_1 > 0$. Note that $(s,y) \mapsto p^{(n)}_s(y,z)$ solves $\partial_t u - L^{(n)}u = 0$ in $(0,\infty)\times \nZd$. Therefore, the weak parabolic Harnack inequality for $L^{(n)}$ is applicable in the time-space cylinder $((R/32)^{\alpha} - (R/64)^{\alpha} , (R/32)^{\alpha} - (R/64)^{\alpha} + 2R^{\alpha}) \times \Bn_{2R}(y_0)$ (see \autoref{thm:PHI}) after enlarging the domain of integration (respectively summation). Then, we obtain by setting $t_0 = c R^{\alpha}$, where $c := 2 + 32^{-\alpha} -64^{-\alpha}-2^{-1-\alpha} > 0$:
\begin{align*}
1 &\le c_2 n^{-d}\sum_{z \in \Bn_{R/16}(y_0)} \dashint_{((R/32)^{\alpha}-(R/64)^{\alpha},(R/32)^{\alpha}-(R/64)^{\alpha}+(R/2)^{\alpha})} (nR/2)^{-d}\sum_{y \in \Bn_{R/2}(y_0)}  p^{(n)}_{s}(y,z) \d s\\
&\le c_3 n^{-d}\sum_{z \in \Bn_{R/16}(y_0)} \inf_{(s,y) \in ((R/32)^{\alpha}-(R/64)^{\alpha} + 2R^{\alpha} - (R/2)^{\alpha},(R/32)^{\alpha}-(R/64)^{\alpha} + 2R^{\alpha}) \times \Bn_{R/2}(y_0)} p^{(n)}_s(y,z) \\
&\le c_3 \inf_{(s,y) \in ((R/32)^{\alpha}-(R/64)^{\alpha} + 2R^{\alpha} - (R/2)^{\alpha},(R/32)^{\alpha}-(R/64)^{\alpha} + 2R^{\alpha}) \times \Bn_{R/2}(y_0)} P^{(n)}_{s} \mathbbm{1}_{\Bn_{R/16}(y_0)} (y)\\
&\le c_3 \inf_{y \in \Bn_{R/2}(y_0)} P^{(n)}_{t_0} \mathbbm{1}_{\Bn_{R/16}(y_0)} (y),
\end{align*}
where $c_2,c_3 > 0$ and we used that $t_0 \in ((R/32)^{\alpha}-(R/64)^{\alpha} + 2R^{\alpha} - (R/2)^{\alpha},(R/32)^{\alpha}-(R/64)^{\alpha} + 2R^{\alpha})$ by definition. We have proved that \eqref{eq:NDLHKB} holds true with $\eps = c_3^{-1}$.

Next, we deduce \eqref{eq:ETE2}: Let now $x_0 \in \nZd$ be arbitrary and $\frac{32\sigma}{n} < R \le 1$. Let $y_0 \in \nZd$ be such that $y_0 \in \Bn_{5R/16}(x_0) \setminus \Bn_{4R/16}(x_0)$. Then by \eqref{eq:NDLHKB} and Markovianity of $(P^{(n)}_t)$ it holds for every $x \in \Bn_{R/2}(y_0)$:
\begin{align*}
1 - P^{(n)}_{t_0} \mathbbm{1}_{\Bn_{R/8}(x_0)}(x) &\ge P^{(n)}_{t_0} \mathbbm{1}_{\Bn_{R/16}(y_0)}(x) \ge \eps.
\end{align*}
Note that $\Bn_{R/8}(x_0) \subset \Bn_{R/2}(y_0)$ by construction. Therefore, by rearranging the above inequality and applying \eqref{eq:ressgmon}, it follows that for every $x \in \Bn_{R/8}(x_0)$:
\begin{align*}
P^{\Bn_{R/8}(x_0)}_{t_0} \mathbbm{1}(x) \le P^{(n)}_{t_0} \mathbbm{1}_{\Bn_{R/8}(x_0)}(x) \le 1 - \eps.
\end{align*}
Using semigroup property and Markovianity of $(P^{\Bn_{R/8}(x_0)}_t)$, we deduce that for every $k \in \N_0$:
\begin{align}
\label{eq:Ptlarget}
P^{\Bn_{R/8}(x_0)}_{s} \mathbbm{1}(x) \le (1-\eps)^k, ~~ s \in [kt_0,kt_0+t_0),
\end{align}
and therefore we obtain that for every $x \in \Bn_{R/8}(x_0)$
\begin{align*}
G^{\Bn_{R/8}(x_0)} \mathbbm{1}(x) = \int_0^{\infty} P^{\Bn_{R/8}(x_0)}_{s} \mathbbm{1}(x) \d s \le c_4 \sum_{k=0}^{\infty} (1-\eps)^k t_0 \le cR^{\alpha},
\end{align*}
where $c_4 > 0$, as desired for \eqref{eq:ETE2}.

We continue with the proof of the second estimate \eqref{eq:ETE1}.
Given $t > 0$ and $\eps > 0$, we define $u(x) = \int_0^t P^{\Bn_{R/8}(x_0)}_s \mathbbm{1}(x) \d s + \eps$ which is an approximation of $G^{\Bn_{R/8}(x_0)}$. It holds that $u > \eps$ and $u$ is a weak solution to $-L^{(n)}u = 1 - P_t^{\Bn_{R/8}(x_0)}\mathbbm{1} \ge 0$ in $\Bn_{R/8}$.  By applying the weak elliptic Harnack inequality (\eqref{eq:EHI}) to $u$, as well as Jensen's inequality, we obtain that
\begin{align}
\label{eq:ETElhelp1}
\inf_{\Bn_{R/32}(x_0)} u \ge c_5 \left(\frac{nR}{32}\right)^{-d} \sum_{x \in \Bn_{R/32}(x_0)} u(x) \ge c_6 (nR)^{d} \left(\sum_{x \in \Bn_{R/32}(x_0)} u^{-1}(x) \right)^{-1},
\end{align}
where $c_5,c_6 > 0$ are constants. Next, we apply \autoref{lemma:logu} and obtain
\begin{align}
\label{eq:ETElhelp2}
\begin{split}
n^{-d} \sum_{x \in \Bn_{R/32}(x_0)} u^{-1}(x) &\le <\mathbbm{1}_{\Bn_{R/8}(x_0)},\tau^2 u^{-1}>\\
&\vspace{-0.6cm}\le \cEn(u,\tau^2 u^{-1}) + <P_t^{\Bn_{R/8}(x_0)} \mathbbm{1} , \tau^2 u^{-1}>\\
&\le c_7 \mu^{(n)} (\Bn_R) R^{-\alpha} + <P_t^{\Bn_{R/8}(x_0)} \mathbbm{1} , \tau^2 u^{-1}>,
\end{split}
\end{align}
where $c_7 > 0$ and $\tau$ is a cut-off function with $\supp(\tau) = \Bn_{R/8}(x_0)$, $\tau \equiv 1$ in $\Bn_{R/32}(x_0)$ and\\ $\max_{i = 1,...,d} \Vert\nabla^{(n)}_i \tau\Vert_{L^{\infty}(\nZd)} \le 2(3R/32)^{-1}$. By combining \eqref{eq:ETElhelp1} and \eqref{eq:ETElhelp2}, we obtain
\begin{align*}
c_8\left(R^{-\alpha} + R^{-d}  <P_t^{\Bn_{R/8}(x_0)} \mathbbm{1} , \tau^2 u^{-1}> \right)^{-1} \le c_6 (nR)^{d} \left(\sum_{x \in \Bn_{R/32}(x_0)} u^{-1}(x) \right)^{-1} \le \inf_{\Bn_{R/32}(x_0)} u
\end{align*}
for some $c_8 > 0$. Finally, note that by \eqref{eq:Ptlarget} and \eqref{eq:ETE2}, the left hand side converges to $c_8 R^{-\alpha}$, as $t \nearrow \infty$. Thus, \eqref{eq:ETE1} follows by taking the limit $\eps \searrow 0$.
\end{proof}

\begin{remark}
Let us point out that the proof of \eqref{eq:ETE1} in \cite{GHH18} Lemma 4.1 is not applicable in our setup since it does not allow for bounded range Markov chains.
\end{remark}

\begin{remark}
One can prove the following near diagonal lower heat kernel bound: There exists $c > 0$ such that for every $c \left( \frac{\theta}{n} \right)^{\alpha} < t \le c$ and $x,y \in \nZd$ with $\vert x-y \vert \le \frac{1}{64 c^{1/\alpha}} t^{1/\alpha} $ it holds 
\begin{align}
p^{(n)}_t(x,y) \ge c t^{-d/\alpha}.
\end{align}
This follows from running a similar argument as in the proof of \eqref{eq:NDLHKB} involving the weak parabolic Harnack inequality for $\widehat{L}^{(n)}$.
\end{remark}

The next result establishes the survival estimate for $(X^{(n)})$.
Its proof is based on the parabolic maximum principle \autoref{prop:pmp} and \autoref{lemma:ETE} and uses the ideas from \cite{GHH18} Lemma 5.6 and Theorem 7.2.1 in \cite{Bos20}.

\begin{lemma}[survival estimate]
\label{lemma:SE}
Assume that \eqref{K1}, \eqref{K2}, \eqref{Poinc}, \eqref{Sob}, \eqref{CTail} hold true for some $\alpha \in (0,2]$, $\sigma > 0$ and $\theta \in (\frac{d}{\alpha},\infty]$. Then there exists $\eps,\delta \in (0,1)$ such that for every $n \in \N$, $\frac{32\sigma}{n} < R \le 1$, $0 < t \le (\delta R)^{\alpha}$, $x_0 \in \nZd$ it holds
\begin{align}
\label{eq:SE}
\inf_{\Bn_{R/32}(x_0)} P^{\Bn_{R/8}(x_0)}_t \mathbbm{1}_{\Bn_{R/8}(x_0)} \ge \eps.
\end{align}
\end{lemma}

\begin{proof}
The goal is to prove that for every $(t,x) \in (0,\infty) \times \Bn_{R/8}(x_0)$:
\begin{align}
\label{eq:SEhelp1}
P^{\Bn_{R/8}(x_0)}_t\mathbbm{1}_{\Bn_{R/8}(x_0)}(x) \ge \frac{G^{\Bn_{R/8}(x_0)}\mathbbm{1}(x) - t}{\Vert G^{\Bn_{R/8}(x_0)} \mathbbm{1}\Vert_{L^{\infty}(\Bn_{R/8}(x_0))}}.
\end{align}
Combining \eqref{eq:SEhelp1} with \autoref{lemma:ETE}, we obtain that for every $(t,x) \in (0,\infty) \times \Bn_{R/32}(x_0)$:
\begin{align*}
P^{\Bn_{R/8}(x_0)}_t\mathbbm{1}_{\Bn_{R/8}(x_0)}(x) \ge \frac{c_2R^{\alpha} - t}{c_1R^{\alpha}}
\end{align*}
for some constants $0 < c_1 < c_2$. By choosing $\delta > 0$ such that $\delta^{\alpha} < \frac{c_2}{2}$, we obtain that for every $0 < t \le (\delta R)^{\alpha}$ and $x \in \Bn_{R/32}(x_0)$ it holds:
\begin{align*}
P^{\Bn_{R/8}(x_0)}_t\mathbbm{1}_{\Bn_{R/8}(x_0)}(x) \ge \frac{c_2R^{\alpha} - c_2 R^{\alpha}/2}{c_1R^{\alpha}} \ge \eps
\end{align*}
for $\eps = c_2/(2c_1)$. As this proves the desired result, it remains to show \eqref{eq:SEhelp1}.\\
The main ingredient in the proof of \eqref{eq:SEhelp1} is the parabolic maximum principle \autoref{prop:pmp}. We will apply it on $(0,T)\times \Bn_{R/8}(x_0)$ for some $T > 0$ to the function $w$ defined by
\begin{align*}
w(t,x):= u(x) - \phi(x)t-\Vert u \Vert_{L^{\infty}(\Bn_{R/8}(x_0))}P^{\Bn_{R/8}(x_0)}_t \mathbbm{1}_{\Bn_{R/8}(x_0)}(x),
\end{align*}
where $u = \int_0^s P_t^{\Bn_{R/8}(x_0)} \mathbbm{1} \d t$ for some fixed $s > 0$, and $\phi \in L^2(\nZd)$ is chosen such that $0 \le \phi \le 1$, $\phi \equiv 1$ on $\Bn_{R/8}(x_0)$ and $\supp(\phi) \subset \Bn_{R/4}(x_0)$.\\
One immediately sees that $w_+(t) \in L^2_c(\Bn_{R/8}(x_0))$ for every $t > 0$, and that $w_+(t) \to 0$ in $L^2(\nZd)$ as $t \searrow 0$. Furthermore, $w$ is a subsolution to $\partial_t u - L^{(n)}u = 0$ in $(0,T)\times \Bn_{R/8}(x_0)$ since for every nonnegative $\psi \in L^2_c(\Bn_{R/8}(x_0))$ and $t > 0$ it holds:
\begin{align*}
<\partial_t w(t),\psi> &+ \cEn(w(t),\psi) = -<\phi + \Vert u \Vert_{L^{\infty}(\Bn_{R/8}(x_0))}\partial_t P^{\Bn_{R/8}(x_0)}_t \mathbbm{1}_{\Bn_{R/8}(x_0)},\psi> + \cEn(w(t),\psi)\\
&= -<\phi,\psi> + \cEn(u,\psi)-t \cEn(\phi,\psi)\\
&- \Vert u \Vert_{L^{\infty}(\Bn_{R/8}(x_0))} \left(<\partial_t P^{\Bn_{R/8}(x_0)}_t \mathbbm{1}_{\Bn_{R/8}(x_0)},\psi> + \cEn(P^{\Bn_{R/8}(x_0)}_t \mathbbm{1}_{\Bn_{R/8}(x_0)},\psi) \right)\\
&\le <\mathbbm{1}_{\Bn_{R/8}(x_0)}-\phi,\psi>-t \cEn(\phi,\psi),
\end{align*}
where we used that $(t,x) \mapsto P^{\Bn_{R/8}(x_0)}_t \mathbbm{1}_{\Bn_{R/8}(x_0)}(x)$ solves $\partial_t u - L^{(n)}u = 0$ in $(0,\infty) \times \Bn_{R/8}(x_0)$ and that $u$ solves $-L^{(n)}u = \mathbbm{1} - P_t^{\Bn_{R/8}(x_0)} \mathbbm{1}$ in $\Bn_{R/8}(x_0)$. 
By the definition of $\phi,\psi$, note that $<\mathbbm{1}_{\Bn_{R/8}(x_0)}-\phi,\psi> \le 0$ and 
\begin{align*}
-\cE^{(n)}(\phi,\psi) = -2n^{\alpha-d} \sum_{x \in \Bn_{R/8}(x_0)}\sum_{y \in \nZd} (1-\phi(y))\psi(x)\Cn(x,y) \le 0.
\end{align*}
Therefore, $w$ is a weak subsolution to $\partial_t u - L^{(n)} u = 0$ in $(0,T)\times \Bn_{R/8}(x_0)$ and the weak the parabolic maximum principle is applicable to $w$. Since $T > 0$ was arbitrary, it follows that $w \le 0$ in $(0,\infty)\times \Bn_{R/8}(x_0)$ and therefore \eqref{eq:SEhelp1} holds true after taking the limit $s \nearrow \infty$. This concludes the proof.
\end{proof}

\begin{remark}
In probabilistic terms \eqref{eq:SE} yields the existence of $\eps,\delta \in (0,1)$ such that for every $\frac{\sigma}{4n} < R \le \frac{1}{8}$ and every $x_0 \in \nZd$:
\begin{align}
\label{eq:stochSE}
\PP^x\left( \tau_{\Bn_{R}(x_0)} \le (\delta R)^{\alpha} \right) \le 1-\eps, ~~\forall x \in \Bn_{R/4}(x_0).
\end{align}
As a consequence, we have that for every $x \in \nZd$:
\begin{align*}
\PP^x\left(\sup_{t \le (\delta R)^{\alpha}} \vert X^{(n)}_t -x \vert > R\right) \le 1-\eps.
\end{align*}
This estimate is weaker than \autoref{thm:tightnessprep} since $\eps > 0$ cannot be arbitrary in \eqref{eq:stochSE}.
\end{remark}
 
In order to establish \autoref{thm:tightnessprep}, we iterate statements of the form \eqref{eq:SE} (resp. \eqref{eq:stochSE}) using the following lemma. We adapt the proof of Lemma 7.3.1 in, \cite{Bos20}, which is based on Lemma 4.6 in \cite{GHH17}. 
A similar iteration of survival estimates is carried out for the proof of Theorem 3.1 in \cite{GHL14} and Theorem 5.7 in \cite{GrHu14}.

\begin{lemma}[iteration lemma]
\label{lemma:SEit}
Assume that \eqref{K1}, \eqref{K2}, \eqref{Poinc}, \eqref{Sob}, \eqref{CTail} hold true for some $\alpha \in (0,2]$, $\sigma > 0$ and $\theta \in (\frac{d}{\alpha},\infty]$. Let $H > 0$, $C \in (1,2)$, $n \in \N$ and $\frac{\sigma}{n} < R \le \frac{1}{C}$ with $\frac{\sigma}{4n} < \frac{(C-1)R}{3} \le \frac{1}{8}$ and $x_0 \in \nZd$. Assume that for some $\gamma_0 \in (0,1)$ and some open set $M \subset \Bn_R(x_0)$:
\begin{align}
\label{eq:SEitass}
1 - P^{\Bn_R(x_0)}_t \mathbbm{1}(x) \le H,\qquad  \forall x \in M,~~ 0 < t \le (\gamma_0 R)^{\alpha}.
\end{align}
Then there exists $\kappa \in (0,1)$, independent of $H,R,C,n,\gamma_0$ and $\gamma \in (0,1)$, independent of $R,n$:
\begin{align*}
1 - P^{\Bn_{CR}(x_0)}_t \mathbbm{1}(x) \le \kappa H, \forall x \in M, ~~ 0 < t \le (\gamma R)^{\alpha}.
\end{align*}
\end{lemma}

\begin{proof}
Let $\eps,\delta_0$ be the $\eps,\delta$ from \autoref{lemma:SE} and $\delta = \min(\gamma_0,\delta_0)$. Let $\psi \in L^2(\nZd)$ be such that $0 \le \psi \le 1$, $\psi \equiv 1$ on $\Bn_{R + \frac{C-1}{3}R}(x_0)$ and $\supp(\psi) \subset \Bn_{R + \frac{2(C-1)}{3}R}(x_0)$. Define for $\beta > 0$ the function
\begin{align*}
u(t,x) = P^{\Bn_R(x_0)}_t \mathbbm{1}(x) - \frac{P^{\Bn_{CR}(x_0)}_t\mathbbm{1}(x)-\eps\psi(x)}{1-\eps} - \beta t \psi(x).
\end{align*}
We prove that the parabolic maximum principle \autoref{prop:pmp} is applicable to $u$ in $(0,T)\times\Bn_R(x_0)$, where $T = (\delta \frac{C-1}{3}R)^{\alpha}$, if $\beta > 0$ is chosen suitably.\\
First, by application of the survival estimate (see \autoref{lemma:SE}) to balls of the form $\Bn_{\frac{C-1}{3}R}(y)$ for $y \in \Bn_{R + \frac{2(C-1)}{3}R}(x_0)$, we obtain that
\begin{align*}
P^{\Bn_{CR}(x_0)}_t \mathbbm{1}(x) \ge \eps, \qquad x \in \Bn_{R + \frac{2(C-1)}{3}R}(x_0), ~~ 0 < t \le T.
\end{align*}
This is a simple consequence of the fact that $\Bn_{CR}(x_0)$ can be covered by the family of balls $\Bn_{\frac{C-1}{3}R}(y)$ as above and since $P_t^{\Bn_{CR}(x_0)} \mathbbm{1}(x) \ge P_t^{\Bn_{\frac{C-1}{3}R}(y)} \mathbbm{1}(x)$ due to \eqref{eq:ressgmon}. By definition of $\psi$, we conclude that for every $t \le T $ on $\Bn_R$ it holds that $P^{\Bn_{CR}(x_0)}_t \mathbbm{1} - \eps\psi \ge 0$. Consequently, we have that $u_+(t) \in L^2_c(\Bn_R(x_0))$ for every $0 < t \le T$.\\
Further, one easily sees that $u_+(t) \to 0$ in $L^2(\Bn_R(x_0))$, as $t \searrow 0$, as a consequence of the strong continuity of $(P^{\Bn_R(x_0)}_t)$ and $(P^{\Bn_{CR}(x_0)}_t)$.\\
It remains to check that $u$ is a subsolution to $\partial_t u - L^{(n)}u = 0$ in $(0,T) \times \Bn_R(x_0)$. We take an arbitrary function $\phi \in L^2_c(\Bn_R(x_0))$ with $\phi \ge 0$ and compute for $t \in (0,T)$:
\begin{align*}
<\partial_t u(t),\phi> + \cEn(u(t),\phi) &= \frac{\eps}{1-\eps}\cEn(\psi,\phi) - \beta<\psi,\phi> - t\beta\cEn(\psi,\phi)\\
&\le \frac{\eps}{1-\eps}\cEn(\psi,\phi) - \beta<1,\phi>,
\end{align*}
where we used that
\begin{align*}
\cEn(\psi,\phi) &= 2n^{\alpha-d}\sum_{x \in \nZd}\sum_{y\in\nZd} (\psi(x)-\psi(y))\phi(x)\Cn(x,y)\\
&= 2n^{\alpha-d}\sum_{x \in \Bn_R}\sum_{y\in\nZd\setminus \Bn_{R + \frac{C-1}{3}R}(x_0)} (1-\psi(y))\phi(x)\Cn(x,y) \ge 0.
\end{align*}
Next, we apply \eqref{CTail} with $r = \frac{C-1}{3}R$ and compute that
\begin{align*}
\cEn(\psi,\phi) \le 2n^{\alpha-d}\sum_{x \in \Bn_R(x_0)}\phi(x)\left(\sum_{y\in\nZd\setminus \Bn_{\frac{C-1}{3}R}(x)} \Cn(x,y)\right) \le c\left(\frac{C-1}{3}R\right)^{-\alpha}<1,\phi>,
\end{align*}
where $c > 0$ is a constant. We choose $\beta = \frac{c\eps}{1-\eps}\left(\frac{C-1}{3}R\right)^{-\alpha}$ and obtain that for $t \in (0,T)$
\begin{align*}
<\partial_t u(t),\phi> + \cEn(u(t),\phi) \le 0,
\end{align*}
as desired. Note that when $\cEn(\psi,\phi) = 0$ (i.e., for $n$ large in the bounded range-case), we can simply choose $\beta = 0$.

Next, we apply the parabolic maximum principle (\autoref{prop:pmp}), which yields that $u \le 0$ in $(0,T) \times \Bn_R(x_0)$. By using the definition of $u$, as well as \eqref{eq:SEitass}, we obtain that for every $0 < t < T = \min(T,(\delta R)^{\alpha})$ it holds:
\begin{align*}
1-H - \beta t \le P_t^{\Bn_R(x_0)} \mathbbm{1} - \beta t \le \frac{P^{\Bn_{CR}(x_0)}_t \mathbbm{1} - \eps}{1 - \eps}, ~~\forall x \in M,
\end{align*}
which is equivalent to
\begin{align*}
1 - P^{\Bn_{CR}(x_0)}_t \mathbbm{1} \le (1-\eps)(H + \beta t), ~~\forall x \in M.
\end{align*}
Finally, we note that if $0 < t < \min\left(T,\frac{H \eps}{2(1-\eps)\beta}\right)$, we obtain that
\begin{align}
\label{eq:SEithelp1}
1 - P^{\Bn_{CR}(x_0)}_t \mathbbm{1}(x) \le \kappa H, ~~\forall x \in M,
\end{align}
where $\kappa = 1 - \frac{\eps}{2}$. If $\beta = 0$, we set $\frac{H \eps}{2(1-\eps)\beta} = \infty$. Note that by the definition of $T$ and $\beta$, we can find a constant $\gamma > 0$, independent of $R$, such that \eqref{eq:SEithelp1} holds for every $0 < t \le (\gamma R)^{\alpha}$. This concludes the proof.
\end{proof}

We are finally in the position to prove \autoref{thm:tightnessprep}.

\begin{proof}[Proof of \autoref{thm:tightnessprep}]

Let $A,B \in (0,1)$ be arbitrary and $x_0 \in \nZd$. Note that it is enough to prove \autoref{thm:tightnessprep} for $A \in (0,\frac{1}{4})$ by inclusion of sets. Our goal is to find $t_0 \in (0,1)$, $N \in \N$ such that for every $0 < R \le 1$ and $n \ge N$:
\begin{align*}
1 - P^{\Bn_{AR}(x_0)}_t \mathbbm{1}(x) \le B, \qquad \forall x \in \Bn_{AR/2}(x_0),~~ 0 < t \le (t_0 R)^{\alpha}. 
\end{align*}
We set $\eta = A/2$ and, for $k \in \N$, we define $R_k = C^k \eta R$, where $C \in (1,2)$ is to be chosen later. First, we have the following trivial estimate for every $n \in \N$, $0 < R \le 1$:
\begin{align*}
1 - P^{\Bn_{R_0}(x_0)}_t \mathbbm{1}(x) \le 1, \qquad \forall x \in \Bn_{R_0}(x_0) = \Bn_{AR/2}(x_0),~~ t > 0.
\end{align*}
We observe that \autoref{lemma:SEit} yields that 
\begin{align*}
1 - P^{\Bn_{R_1}(x_0)}_t \mathbbm{1}(x) \le \kappa, \qquad  \forall x \in \Bn_{AR/2}(x_0),~~ 0 < t \le (\gamma R)^{\alpha}
\end{align*}
for some $\gamma,\kappa \in (0,1)$ in case $\frac{\sigma}{4n} \le \frac{(C-1)AR}{6}$. After having chosen $C \in (1,2)$, we will determine $N \in \N$ (depending on $R$) such that this condition is satisfied for every $n \ge N$. By iterating the above line, we obtain that for every $k \in \N$ (as long as $R_k \le \frac{1}{4}$): 
\begin{align*}
1 - P^{\Bn_{R_k}(x_0)}_t \mathbbm{1}(x) \le \kappa^k, \qquad \forall x \in \Bn_{AR/2}(x_0),~~ 0 < t \le (\gamma_k R)^{\alpha}
\end{align*}
for some $\gamma_k \in (0,1)$ that depend on $C$ but not on $R$. We want to choose $k \in \N$ and $C \in (1,2)$ such that (i) $\kappa^k \le B$ and (ii) $C^k \eta R \le AR$ hold true. By defining $k := \left\lfloor \frac{\log(A/\eta)}{\log(C)} \right\rfloor = \left\lfloor \frac{\log(2)}{\log(C)} \right\rfloor$, we guarantee (ii). Note that by definition, $k \nearrow \infty$ as $C \searrow 1$. Therefore, we can choose $C-1 > 0$ small enough, such that (i) holds, namely we choose $C = \frac{1}{2}\left(2^{\frac{\log \kappa}{\log B}} + 1\right)$. This yields
\begin{align*}
1 - P^{\Bn_{AR}(x_0)}_t \mathbbm{1}(x) \le 1 - P^{\Bn_{R_k}(x_0)}_t \mathbbm{1}(x) \le \kappa^k \le B, \qquad \forall x \in \Bn_{AR/2}(x_0),~~ 0 < t \le (\gamma_k R)^{\alpha}.
\end{align*}
Upon our definition of $C$, it is guaranteed that $\frac{\sigma}{4n} \le \frac{(C-1)AR}{6}$ is satisfied for every $n \ge N := \left\lceil \frac{3\sigma}{AR_k}\left(2^{\frac{\log \kappa}{\log B}} -1\right)^{-1}\right\rceil$. Therefore, the desired result holds true with the choice $t_0 = \gamma_k$.
\end{proof}

\section{Convergence}
\label{sec:convergence}

Given $T > 0$, $x \in \R^d$ and a sequence $(x_n) \subset \nZd$ with $x_n \to x$, our goal is to prove that - under suitable assumptions - the sequence of $\PP^{x_n}$-laws of $(X^{(n)}_t)_{t \in [0,T]}$ converges weakly, with respect to the $D([0,T];\R^d)$-topology, to a probability $\PP^x$ and to identify the limiting process $X$ by providing the associated bilinear form.

The following two theorems collect statements that outline the path towards the desired weak convergence of $(\Xn)$. First, tightness of the laws of $\Xn$ is proved and then, weak convergence of $(\Xn)$ along sub-subsequences is established. These results are contained in \autoref{thm:tightness} and follow from \autoref{cor:regest} and \autoref{thm:tightnessprep}. In order to guarantee convergence of the full sequence of laws of $(\Xn)$ we need to show that the limits in \autoref{thm:tightness} are independent of their respective subsequences. We will do so by proving that all limits correspond to the same bilinear form. In this respect, \autoref{thm:convergence} provides a suitable criterion for the desired convergence in terms of the corresponding bilinear forms.

Our proofs follow the technique that was developed in \cite{StZh97}, \cite{BaKu08}, \cite{HuKa07}, \cite{BKU10}, \cite{DeKu13}. Nevertheless, we clarify some of the arguments since \autoref{thm:tightnessprep} and \autoref{cor:regest} differ slightly from their counterparts in the aforementioned articles.

Given functions $f : \R^d \to \R$, $g : \nZd \to \R$, we define the restriction operator to $\nZd$ by $R^{(n)}f(x) = f(x)$ for every $x \in \nZd$ and the extension operator $E^{(n)}g(x) = g([x]_n)$ for every $x \in \R^d$, where $[x]_n = (\lfloor nx_i \rfloor/n)_{i=1}^d \in \nZd$.

\begin{theorem}
\label{thm:tightness}
Assume that \eqref{K1}, \eqref{K2}, \eqref{Poinc}, \eqref{Sob}, \eqref{CTail}, \eqref{CTail2} hold true for some $\alpha \in (0,2]$, $\sigma > 0$ and $\theta \in (\frac{d}{\alpha},\infty]$. Let $(x_n)_n \subset \nZd$ and $x \in \R^d$ with $x_n \to x$ and $T > 0$. Then, the $\PP^{x_n}$-laws of $(\Xn)$ are tight in $D([0,T];\R^d)$.
Moreover, for every subsequence $(n_j) \subset \N$ there exists a further subsequence $(n_{j_k}) \subset (n_j)$ such that
\begin{itemize}
\item[(i)] For each $f \in C_c(\R^d)$ and $\lambda > \lambda_0$, $(E^{(n_{j_k})}(P^{(n_{j_k})}_t R^{(n_{j_k})}(f)))_k$ and $(E^{(n_{j_k})}(U_{\lambda}^{(n_{j_k})} R^{(n_{j_k})}(f))_k$ converge uniformly on compact subsets.
\item[(ii)] Write $P_t f := \lim_{k \to \infty} E^{(n_{j_k})}(P^{(n_{j_k})}_t R^{(n_{j_k})}(f))$. Then $P_t$ is linear for every $t$ and $(P_t)$ is a semigroup on $C_c(\R^d)$, belonging to a strong Markov process on $\R^d$.
\item[(iii)] The $\PP^{x_{n_{j_k}}}$-laws of $(X^{(n_{j_k})}_t)_{t \in [0,T]}$ converge weakly, with respect to the $D([0,T];\R^d)$-topology, to a probability $\PP^x$.
\end{itemize}
\end{theorem}

In the statement of the theorem, $\lambda_0 > 0$ denotes the constant from \autoref{lemma:elementary}.

\begin{proof}
Note that from \autoref{thm:tightnessprep} it follows that for every $A > 0$ and $B \in (0,1)$ there is a constant $\gamma > 0$ and $N \in \N$ such that for every  $x \in \nZd$ and every $n \ge N$ it holds:
\begin{align}
\label{eq:tighthelp1}
\PP^{x}\left( \tau_{\Bn_A(x)} \le \gamma \right) \le B.
\end{align}
Now, let $T > 0$, $x \in \R^d$ and $(x_n)_n \subset \nZd$ with $x_n \to x$. Moreover, let $\tau_n \in [0,T]$ be a sequence of stopping times for $X^{(n)}$, $(\delta_n) \subset [0,1]$ with $\delta_n \to 0$ . Then, by \eqref{eq:tighthelp1} and the strong Markov property it follows for every $n \ge N$ with $\delta_n \le \gamma$:
\begin{align*}
\PP^{x_n}\left( \vert X^{(n)}_{\tau_n+\delta_n} - X^{(n)}_{\tau_n} \vert > A \right) = \PP^{X^{(n)}_{\tau_n}}\left( \vert X^{(n)}_{\delta_n} - X^{(n)}_{0} \vert > A \right) \le \sup_{x \in \nZd} \PP^{x}\left(\tau_{\Bn_A(x)} \le \gamma \right) \le B.
\end{align*}
This verifies condition (A) in \cite{Ald78}. Moreover, tightness of $\left(\sup_{t \in [0,T]} \vert X^{(n)}_t - X^{(n)}_{t_-}\vert\right)_n$ follows from the L\'evy system formula for $(X^{(n)})$ (see \autoref{lemma:LSformula}), which implies that for $A > 1$:
\begin{align*}
\PP^x\left( \sup_{t \in [0,T]} \vert X^{(n)}_t - X^{(n)}_{t_-}\vert > A \right) &\le \EE^x \left( \sum_{t \le T} \mathbbm{1}_{\{ \vert X^{(n)}_t - X^{(n)}_{t_-}\vert > A \}}(t) \right)\\
&= \EE^x \left( \int_{0}^{T} \sum_{y \in \nZd}\mathbbm{1}_{\{ \vert X^{(n)}_t - y \vert > A \}}(t) n^{\alpha} \Cn(X^{(n)}_t,y) \d t \right)\\
&\le T \sup_{x \in \nZd} n^{\alpha} \sum_{y \in \nZd : \vert x-y \vert > A} \Cn(x,y) \le cT  A^{-\delta}
\end{align*}
by \eqref{CTail2}. Since $x_n \to x$, it follows that the laws of $(X^{(n)})$ are tight in $D([0,T];\R^d)$.

Now, let $(n_j) \subset \N$ be a subsequence. We prove the existence of a further subsequence such that (i), (ii), (iii) hold true. First, given $\lambda > \lambda_0$ and $f \in C_c(\R^d)$, one can deduce from \eqref{eq:resHRE} in \autoref{cor:regest} that the family $(E^{(n_{j})}(U_{\lambda}^{(n_{j})} R^{(n_{j})}(f))$ is equicontinuous and equibounded. Therefore, convergence of the family along a subsequence follows from the Arzel\` a-Ascoli theorem.\\
Let $(t_i)_{i \in \N} \subset (0,\infty)$ and $(f_m)_{m \in \N} \subset C_c(\R^d)$ be dense. Note that \eqref{eq:sgHRE} of \autoref{cor:regest} implies that the family $(E^{(n_j)}(P^{(n_j)}_{t_i} R^{(n_j)}(\widetilde{f}_m) ))_{j,i,m}$ is equicontinuous and equibounded, where we define $\widetilde{f}_m := f_m/\Vert f_m \Vert_{L^{\infty}(\nZd)}$. Again, by the Arzel\`a-Ascoli theorem one can extract a subsequence $(n_{j_k}) \subset (n_j)$ such that $(E^{(n_{j_k})}(P^{(n_{j_k})}_{t_i} R^{(n_{j_k})}(\widetilde{f}_m) ))_{k}$ converges uniformly on compacts as $k \to \infty$. We write $P_{t_i} \widetilde{f}_m$ for the limit object and point out that that it can be extended to all $t > 0$ using the same argument, as in \cite{HuKa07}: Find a subsequence $(t_{i_l})_l \subset (t_i)_i$ such that $t_{i_l} \to t$ as $l \to \infty$ and prove that $(E^{(n_{j_k})}(P^{(n_{j_k})}_{t} R^{(n_{j_k})}(\widetilde{f}_m) ))_{k}$ is a Cauchy-sequence with respect to uniform convergence on compacts, using that
\begin{align*}
&\vert E^{(n_{j_k})}(P^{(n_{j_k})}_{t} R^{(n_{j_k})}(\widetilde{f}_m) ) - E^{(n_{j_k})}(P^{(n_{j_k})}_{t_{i_l}} R^{(n_{j_k})}(\widetilde{f}_m) ) \vert \to 0, ~~\text{as } k,l \to \infty,\\
&\vert E^{(n_{j_k})}(P^{(n_{j_k})}_{t_{i_l}} R^{(n_{j_k})}(\widetilde{f}_m) ) - E^{(n_{j_{k'}})}(P^{(n_{j_{k'}})}_{t_{i_l}} R^{(n_{j_{k'}})}(\widetilde{f}_m) ) \vert \to 0, ~~\text{as } k,k' \to \infty,\\
&\vert E^{(n_{j_{k'}})}(P^{(n_{j_{k'}})}_{t_{i_l}} R^{(n_{j_{k'}})}(\widetilde{f}_m) ) - E^{(n_{j_{k'}})}(P^{(n_{j_{k'}})}_{t} R^{(n_{j_{k'}})}(\widetilde{f}_m) ) \vert \to 0, ~~\text{as } k',l \to \infty,
\end{align*}
where all convergence statements above are meant to be in the sense of uniform convergence on compacts.
While the second convergence result is already known from above, the first and third line follow from estimate \eqref{eq:sgHRE2} of \autoref{cor:regest}. Therefore, the limit $P_t \widetilde{f}_m$ exists uniformly on compacts for every $t > 0$. By density of $(f_m) \subset C_c(\R^d)$, we have proved the desired convergence
result in (ii). Following the arguments from \cite{HuKa07}, one can establish that $(P_t)$ extends to a semigroup on $C_c(\R^d)$ and therefore is associated to a strong Markov process on $\R^d$. This proves (ii). (iii) follows from standard arguments (see \cite{HuKa07}, \cite{BaKu08}): While tightness yields that the laws of $(X^{(n_{j_k})})$ are precompact in the sense that every subsequence of the laws must weakly converge along some further subsequence, properties (i), (ii) guarantee that the weak limit is independent of the actual subsequence since its finite dimensional distributions are determined by $(P_t)$, and hence coincide. This implies that the laws of $(X^{(n_{j_k})})$ already converge.
\end{proof}

\begin{theorem}
\label{thm:convergence}
Assume that \eqref{K1}, \eqref{K2}, \eqref{Poinc}, \eqref{Sob}, \eqref{CTail}, \eqref{CTail2} hold true for some $\alpha \in (0,2]$, $\sigma > 0$ and $\theta \in (\frac{d}{\alpha},\infty]$. Let $(x_n)_n \subset \nZd$ and $x \in \R^d$ with $x_n \to x$ and $T > 0$. Let $(\cE,\mathcal{F})$ be a regular lower bounded semi-Dirichlet form on $L^2(\R^d)$ with core $\mathcal{F}_0 \subset \mathcal{F}$ and $\cE_{\lambda_0} \ge 0$. Assume that for every $\lambda > \lambda_0$, $f \in C_c(\R^d)$, $g \in \mathcal{F}_0$ and every sequence $(n_j) \subset \N$ such that $(E^{(n_{j})}(U_{\lambda}^{(n_{j})} R^{(n_{j})}(f)))_{j}$ converges uniformly on compact subsets, the following holds:
\begin{itemize}
\item[(a)] $H := \lim_{j \to \infty} E^{(n_{j})}(U_{\lambda}^{(n_{j})} R^{(n_{j})}(f)) \in \mathcal{F}$,
\item[(b)] there exists a further subsequence $(n_{{j_k}}) \subset (n_{j})$ such that for every $g \in \mathcal{F}_0$, it holds
\begin{align}
\label{eq:convergenceb}
\cE^{(n_{{j_k}})}(U_{\lambda}^{(n_{{j_k}})} R^{(n_{{j_k}})} f,R^{(n_{{j_k}})} g) \to \cE(H,g), ~~\text{as } k \to \infty.
\end{align}
\end{itemize}
Then the full sequence of $\PP^{x_{n}}$-laws of $(\Xn_t)_{t \in [0,T]}$ converges weakly to a probability $\PP^x$ with respect to the $D([0,T];\R^d)$-topology. Write $X$ for the canonical process on $D([0,T];\R^d)$, then the process $(X,\PP^x)$ is the Markov process associated with the form $(\cE,\mathcal{F})$.
\end{theorem}

\begin{proof}
Due to (iii) of \autoref{thm:tightness}, we know that given any subsequence $(n_j) \subset \N$, the $\PP^{x_{n_j}}$ laws of $(X^{(n_j)})_j$ converge weakly along some further subsequence $(n_{j_k})$ to a probability $\PP^x$. In order for convergence of the full sequence $(\Xn)$ to hold true, we need to prove that the limit of $(X^{(n_{j_k})})$ is independent of the given subsequence $(n_{j_k})$. Write $u_{n_{j_k}} := U_{\lambda}^{(n_{j_k})} R^{(n_{j_k})}(f)$ and $H := \lim_{k \to \infty} E^{(n_{j_k})} u_{n_{j_k}}$ and assume that (a), (b) hold true. It follows that
\begin{align}
\label{eq:convergencehelp1}
\begin{split}
\cE(H,g) &=\lim_{l \to \infty} \cE^{(n_{j_{k_l}})}(u_{n_{j_{k_l}}},R^{(n_{{j_{k_l}}})} g)\\
&=\lim_{l \to \infty} <R^{(n_{{j_{k_l}}})} f,R^{(n_{{j_{k_l}}})} g>_{L^2(n_{j_{k_l}}^{-1}\Z^d)} - \lambda <u_{n_{j_{k_l}}},R^{(n_{{j_k}})} g>_{L^2(n_{j_{k_l}}^{-1}\Z^d)}\\
&= (f,g)_{L^2(\R^d)} - \lambda(H,g)_{L^2(\R^d)}
\end{split}
\end{align}
for every $g \in \mathcal{F}_0$, where $(n_{j_{k_l}}) \subset (n_{j_k})$ denotes the subsequence from (b) and we used the fact that $u_{n_{j_k}}$ is the $\lambda$-resolvent for $\cE^{(n_{j_{k_l}})}$ and that $(E^{(n_{j_k})} u_{n_{j_k}},g)_{L^2(\R^d)} \to (H,g)_{L^2(\R^d)}$. The latter is due to the fact that $(E^{(n_{j_k})} u_{n_{j_k}})_k$ converges uniformly on compacts, $g$ has compact support, and dominated convergence. Using that $\mathcal{F}_0$ is dense in $\mathcal{F}$ with respect to the norm induced by $\cE^{s}(\cdot,\cdot) + \lambda_0\Vert \cdot \Vert^2_{L^2(\R^d)}$, \eqref{eq:convergencehelp1} holds for every $g \in \mathcal{F}$ and therefore we identify $U_{\lambda}f = H \in \mathcal{F}$, i.e., $H$ is the $\lambda$-resolvent of $f$ for $\cE$. Here, $\cE^s := \frac{1}{2}(\cE(u,v) + \cE(v,u))$. Therefore, the limit $H$ does not depend any more on the choice of $(n_j)$, so we conclude that $U^{(n)}_{\lambda} (R^{(n)} f) \to H = U_{\lambda} f$. Thus, also the bilinear form corresponding to the limit process of $(X^{(n_j)})_j$ is uniquely determined. This concludes the proof.
\end{proof}

It remains to verify conditions (a), (b) in order to guarantee weak convergence of the full sequence $(\Xn_t)_{t \in[0,T]}$. In particular, this amounts to proving that all limits in \eqref{eq:convergenceb} coincide. We will do so by providing the limiting bilinear form $\cE$ being governed by coefficient functions $(a_{i,j})_{i,j = 1}^d, (b_i)_{i = 1}^d$ on $\R^d$ in the case $\alpha = 2$ and a jumping kernel $K$ on $\R^d \times \R^d$ when $\alpha \in (0,2)$, which will be determined uniquely through the family of conductances $(\Cn)$. The existence of such functions will be posed as a separate assumption.
We treat the two cases $\alpha = 2$ and $\alpha \in (0,2)$ separately in the following two sections.

\subsection{Approximation of strongly local forms}

In this section, we assume that \eqref{K1}, \eqref{K2}, \eqref{Poinc} and \eqref{CTail} hold true with $\alpha = 2$, $\sigma > 0$ and $\theta \in (\frac{d}{2},\infty]$. Furthermore, we assume that $\Xn$ is comparable to the nearest neighbor random walk (NNRW) in the following sense: \\
There exists $B > 0$ such that for every $u \in L^2(\nZd)$:
\begin{align}
\label{eq:NNRWcomp}
\cEns(u,u) \ge B \cE^{(n)}_{NN}(u,u).
\end{align}
Here, the NNRW is defined through the conductances $NN : \nZd \times \nZd \to [0,\frac{1}{2}]$, given by
\begin{align*}
NN(x,y) &:= \frac{1}{2}\mathbbm{1}_{\lbrace\vert x-y \vert = 1/n\rbrace}(x,y),\\
\cE^{(n)}_{NN}(u,u) &:= n^{2-d}\sum_{x \in \nZd} \sum_{y \in \nZd} (u(x)-u(y))^2 NN(x,y).
\end{align*}

\begin{remark}
A sufficient condition for comparability to the NNRW (see \eqref{eq:NNRWcomp}) is the existence of $\delta > 0$ and $N \in \N$ such that for any pair $(x,y) \in \nZd \times \nZd$ with $\vert x-y \vert = \frac{1}{n}$ there exist $k \le N$, $x_0 = x, x_1, \dots, x_k = y \in \nZd$ such that $\Cns(x_i,x_{i+1}) \ge \delta$ for every $i \in \{0,\dots,k\}$.
\end{remark}

Moreover, we assume that $\Xn$ is of bounded range, i.e., that there exists a constant $C > 0$ such that $\Cn(x,y) = 0$ if $\vert x-y \vert \ge C/n$.

\begin{remark}
\begin{itemize}
\item[(i)] We recall that if $\Xn$ is of bounded range, then \eqref{eq:suffCTailbdrange} is a sufficient condition for \eqref{CTail}, \eqref{CTail2}.
\item[(ii)] It is possible to drop assumption \eqref{CTail2} completely because all statements of \autoref{thm:tightness}, \autoref{thm:convergence} remain valid due to assumption \eqref{CTail2} being trivial for $n > C$. Moreover, \eqref{Sob} follows from the global Sobolev inequality for $\cEn_{NN}$ and \eqref{eq:cutoff}.
\end{itemize}
\end{remark}

We require some additional notation (see \cite{DeKu13}) in order to introduce the connection between $(\Cn)$ and the coefficient functions $a_{i,j},b_i$ which determine the limiting form $(\cE,H^1(\R^d))$:
\begin{itemize}
\item First, recall that $\nabla^{(n)}_i u(x) := n(u(x+e_i/n)-u(x))$ for any function $u : \nZd \to \R$.
\item Define $\mathcal{P}(x,y)$ as the set of shortest nearest neighbor paths (SNNP) in $\nZd$ from $x \in \nZd$ to $y \in \nZd$ and set
\begin{align*}
P^{x,y}(w,z) = \frac{1}{\vert \mathcal{P}(x,y)\vert} \sum_{\sigma \in \mathcal{P}(x,y)} \mathbbm{1}_{\lbrace \sigma = (x = \sigma_0,\dots,y = \sigma_l) : \exists k \le l : w = \sigma_{k-1},z=\sigma_{k} \rbrace}(\sigma), ~~w,z \in \nZd,
\end{align*}
i.e., $P^{x,y}(w,z)$ is the ratio of SNNP from $x$ to $y$ using the edge $(w,z)$ to all SNNP from $x$ to $y$.
\end{itemize}

The quantity $P^{x,y}(w,z)$ is motivated by the following two useful identities (see also \cite{DeKu13}):

\begin{lemma}[see Lemma 5.1 in \cite{BKU10}, p.138 in \cite{BKU10}]
\label{lemma:BKU}
Let $f \in L^2(\nZd)$. Then, for every $x,y \in \nZd$, the following identities hold true:
\begin{align}
\label{eq:BKU}
f(x) - f(y) &= \frac{1}{n} \sum_{i=1}^d \sum_{z \in \nZd} \left( P^{x,y}(z + e_i/n,z)-P^{x,y}(z,z+e_i/n))\right) \nabla^{(n)}_i f(z),\\
\label{eq:BKUhelp}
n (x_i-y_i) &= \sum_{z \in \nZd} \left(P^{x,y}(z+e_i/n,z) + P^{x,y}(z,z+e_i/n) \right), \qquad i = 1,\dots,d.
\end{align}
\end{lemma}

\begin{itemize}
\item For $i,j \in \{1,\dots,d\}$, $z,w \in \nZd$ we define as in \cite{DeKu13}:
\begin{align*}
G^{(n)}_{i,j}(z,w) := \sum_{x \in \nZd} \sum_{y \in \nZd}& (P^{x,y}(z + e_i/n,z)-P^{x,y}(z,z+e_i/n))\\
&(P^{x,y}(w + e_j/n,w)-P^{x,y}(w,w+e_j/n))\Cns(x,y).
\end{align*}
\item We introduce a similar quantity for $\Cna$, namely for $i \in \{1,\dots,d\}$, $x,z \in \nZd$:
\begin{align*}
H^{(n)}_i(x,z) := \sum_{y \in \nZd} (P^{x,y}(z + e_i/n,z)-P^{x,y}(z,z+e_i/n))\Cna(x,y).
\end{align*}
\item For $z \in \nZd$, we define 
\begin{align*}
F^{(n)}_{i,j}(z) = \sum_{w \in \nZd} G^{(n)}_{i,j}(w,z), \qquad B^{(n)}_i(z) = n \sum_{x \in \nZd} H^{(n)}_i(x,z).
\end{align*}
We abuse notation and write $F^{(n)}_{i,j}(z) := E^{(n)}F^{(n)}_{i,j}(z), B^{(n)}_i(z) := E^{(n)}B^{(n)}_i(z)$.
\end{itemize}

By a straightforward computation, similar to (5.1) in \cite{BKU10}, one verifies that
\begin{align*}
\cEn(f,g) &= n^{2-d}\sum_{x \in \nZd}\sum_{y \in \nZd}(f(x)-f(y))(g(x)-g(y))\Cns(x,y)\\
&+ 2n^{2-d}\sum_{x \in \nZd}\sum_{y \in \nZd}(f(x)-f(y))g(x)\Cna(x,y)\\
&= n^{-d}\sum_{i,j =1}^d \sum_{z \in \nZd} \sum_{w \in \nZd}\nabla^{(n)}_i f(z)\nabla^{(n)}_j g(w)G^{(n)}_{i,j}(w,z)\\
&+ 2n^{1-d} \sum_{i = 1}^d \sum_{x \in \nZd} \sum_{z \in \nZd} \nabla^{(n)}_i f(z)g(x) H^{(n)}_i(x,z).
\end{align*}

We are now ready to formulate an assumption on the coefficient functions $(a_{i,j})_{i,j = 1}^d, (b_i)_{i = 1}^d$ of the limiting bilinear form:

\begin{assumption}
\label{ass:local}
There exist $a_{i,j} , b_i : \R^d \to \R$, $i,j \in \{1,\dots,d\}$ with $a_{i,j} = a_{j,i}$, such that $\Vert F^{(n)}_{i,j} - a_{i,j}\Vert_{L^1_{loc}(\R^d)} \to 0$ and $\Vert B^{(n)}_i - b_i \Vert_{L^1_{loc}(\R^d)} \to 0$ as $n \to \infty$, and moreover
$a_{i,j}$ is uniformly elliptic, bounded, and $b_i$ satisfies $|b_i|^2 \in L^{\theta}(\R^d)$.
\end{assumption}

Note that \autoref{ass:local} is sufficient for $(\cE,H^1(\R^d))$ defined by
\begin{align}
\label{eq:limitforml}
\begin{split}
\cE(f,g) &= \cE^{a_{i,j}}(f,g) + \cE^{b_i}(f,g)\\
&= \int_{\R^d} a_{i,j}(x)\partial_i f(x)\partial_j g(x) \d x + 2\int_{\R^d} b_i(x)\partial_i f(x)g(x) \d x
\end{split}
\end{align}
is a regular lower bounded semi-Dirichlet form on $L^2(\R^d)$ (see \cite{MaRo92}, p.30-35 in \cite{Osh13}).\\
Moreover, \autoref{ass:local} implies that $F^{(n)}_{i,j} \to a_{i,j}$, $B_i^{(n)} \to b_i$ in measure on each compact set and that a subsequence converges pointwise a.e.

\begin{remark}
Note that $F^{(n)}_{i,j} \in L^{\infty}(\nZd)$ and $|B^{(n)}_i|^2 \in L^{\theta}(\nZd)$ with norms uniform in $n$.
Boundedness of $F_{i,j}^{(n)}$ can be proved as follows: For every $n \in \N$ $z \in \nZd$:
\begin{align*}
F_{i,j}^{(n)}(z) &\le n^2 \sum_{x \in \nZd : \vert x-z\vert \le \frac{C}{n}} \sum_{y \in \nZd : \vert x-y \vert \le \frac{C}{n}} \vert x-y \vert^2 \Cns(x,y)\\
&\le C^d \sup_{x \in \nZd} n^2 \sum_{y \in \nZd : \vert x-y \vert \le \frac{C}{n}} \vert x-y \vert^2 \Cns(x,y) \le c_1 < \infty
\end{align*}
for some $c_1 > 0$, due to \eqref{eq:BKUA3}. Moreover, we used \eqref{eq:BKUhelp} and the fact that $\Xn$ is of bounded range.
Besides, we have for every $z \in \nZd$:
\begin{align*}
&B^{(n)}_i(z) \le n^2 \sum_{x \in \nZd : \vert x-z \vert \le \frac{C}{n}} \sum_{y \in \nZd : \vert x-y \vert \le \frac{C}{n}} \vert x-y \vert \vert \Cna(x,y) \vert\\
&\le \left[\sum_{x: \vert x-z \vert \le \frac{C}{n}} \hspace{-0.1cm}\left(n^2 \hspace*{-2ex}\sum_{y: \vert x-y \vert \le \frac{C}{n}} \vert x-y \vert^2 J^{(n)}(x,y)\right)^{\frac{2\theta}{2\theta-1}}\right]^{1 - \frac{1}{2\theta}}\hspace{-0.2cm} \left[ \sum_{x: \vert x-z \vert \le \frac{C}{n}}\left(n^2 \hspace*{-2ex}\sum_{y: \vert x-y \vert \le \frac{C}{n}} \frac{\vert \Cna(x,y)\vert^2}{J^{(n)}(x,y)}\right)^{\theta} \right]^{\frac{1}{2\theta}}\\
&\le c_2 C^{d\left(1 - \frac{1}{2\theta} \right)} \left[ \sum_{x: \vert x-z \vert \le \frac{C}{n}}\left(n^2 \hspace*{-2ex}\sum_{y: \vert x-y \vert \le \frac{C}{n}} \frac{\vert \Cna(x,y)\vert^2}{J^{(n)}(x,y)}\right)^{\theta} \right]^{\frac{1}{2\theta}}
\end{align*}
for some $c_2 > 0$, due to \eqref{eq:BKUA3}, \eqref{eq:BKUhelp} and \eqref{K1}. Consequently,
\begin{align*}
\Vert (B_i^{(n)})^2\Vert_{L^{\theta}(\nZd)}^{\theta} &\le c_3 n^{-d} \sum_{z \in \nZd} \sum_{x: \vert x-z \vert \le \frac{C}{n}}\left(n^2 \hspace*{-2ex}\sum_{y: \vert x-y \vert \le \frac{C}{n}} \frac{\vert \Cna(x,y)\vert^2}{J^{(n)}(x,y)}\right)^{\theta}\\
&\le c_4 C^d \left \Vert n^2 \sum_{y \in \nZd} \frac{|\Cna(\cdot,y)|^2}{|J^{(n)}(\cdot,y)|} \right\Vert_{L^{\theta}(\nZd)}^{\theta} \le c_5
\end{align*}
for some $c_3,c_4, c_5 > 0$.

\end{remark}

The following theorem is the main result of this article in case $\alpha = 2$:

\begin{theorem}[central limit theorem]
\label{thm:CLTl}
Assume that $(\Cn)$ satisfies \eqref{K1}, \eqref{K2}, \eqref{Poinc} and \eqref{CTail}  with $\alpha = 2$, $\sigma > 0$ and $\theta \in (\frac{d}{2},\infty]$. Furthermore, we assume that $\Xn$ is comparable to the NNRW and of bounded range.
Assume that there are $a_{i,j}$, $b_i$ such that \autoref{ass:local} holds true. Let $(x_n)_n \subset \nZd$ and $x \in \R^d$ with $x_n \to x$ and $T > 0$. Then the $\PP^{x_{n}}$-laws of $(\Xn_t)_{t \in [0,T]}$ converge weakly, with respect to the $D([0,T];\R^d)$-topology to the $\PP^x$-law of $(X_t)_{t \in [0,T]}$, where $(X,\PP^x)$ is the Markov process corresponding to the form $(\cE,H^1(\R^d))$.
\end{theorem}

\begin{proof}
The proof uses the same ideas as in \cite{BKU10}. We need to verify properties (a), (b) from \autoref{thm:convergence}. Let $f \in C_c(\R^d)$, $g \in C_c^2(\R^d)$, $(n_j) \subset \N$. We denote $u_{n_j} = U^{(n_j)}_{\lambda} R^{(n_j)} f$ and $H = \lim_{j \to \infty} E^{(n_j)} u_{n_j}$.
The proof of (a) works exactly as in \cite{BKU10}. Note that uniform boundedness of $(\cE^{(n_{j_k})}_{NN}(E^{(n_{j_k})}u_{n_{j_k}},E^{(n_{j_k})}u_{n_{j_k}}))_k$ follows from G\r{a}rding's inequality for $\cEn$ (see \eqref{eq:Garding-markov}), as well as \eqref{eq:resolventestimate}, and \eqref{eq:NNRWcomp}.

It therefore remains to show (b), i.e., that there exists a further subsequence $(n_{j_k}) \subset (n_j)$ such that $\cE^{(n_{j_k})}(u_{n_{j_k}},R^{(n_{j_k})} g) \to \cE(H,g), \text{ as } k \to \infty$. For simplicity of notation, we will assume that already $(u_n)$ converges. It was shown already in \cite{BKU10} that $\cEns(u_n,R^{(n)} g) \to \cE^{a_{i,j}}(H,g)$ converges along some subsequence. Note that \eqref{CTail} implies \eqref{eq:BKUA3}, i.e.,  assumption (A3) in \cite{BKU10}, which is required for their argument to work. Therefore, it remains for us to prove that also 
\begin{align}
\label{eq:CLTlhelp1}
\cEna(u_n,R^{(n)} g) \to \cE^{b_i}(H,g), \text{ as } n \to \infty
\end{align}
along some subsequence. Let us denote $K^{(n)} = \supp(g) \cap \nZd$. We can write
\begin{align*}
\cEna(u_n,R^{(n)} g) &= 2n^{2-d} \sum_{x \in \nZd} \sum_{y \in \nZd} (u_n(x)-u_n(y))g(x)\Cna(x,y)\\
&= 2n^{1-d} \sum_{i = 1}^d \sum_{z \in \nZd} \sum_{x \in K^{(n)} : \vert x-z \vert \le \frac{C}{n}} \nabla^{(n)}_i u_n(z) g(x) H^{(n)}_i(x,z)\\
&= 2n^{-d} \sum_{i = 1}^d\sum_{z \in \nZd} \Bigg(B^{(n)}_i(z)\nabla^{(n)}_i u_n(z)g(z)\\
&\qquad\qquad\qquad+ \sum_{x \in K^{(n)} : \vert x-z \vert \le \frac{C}{n}} n H^{(n)}_i(x,z)\nabla^{(n)}_i u_n(z) (g(x)-g(z)) \Bigg)\\
&=: I^{(n)}_1 + I^{(n)}_2.
\end{align*}
We separately analyze the summands $I^{(n)}_1$ and $I^{(n)}_2$. It will turn out that $I^{(n)}_1 \to \cE^{b_i}(H,g)$ (up to a subsequence), while $I^{(n)}_2 \to 0$, yielding \eqref{eq:CLTlhelp1}, as desired.\\
For $I^{(n)}_2$, we argue as follows:
\begin{align*}
\vert I^{(n)}_2\vert &= \left\vert2n^{1-d} \sum_{i = 1}^d\sum_{z \in \nZd}\sum_{x \in K^{(n)} : \vert x-z \vert \le \frac{C}{n}} H^{(n)}_i(x,z)\nabla^{(n)}_i u_n(z) (g(x)-g(z))\right\vert\\
&\le c_1\left(\sup_{\vert z-z'\vert \le \frac{1}{n}} \vert u_n(z) - u_n(z') \vert \right)\Vert \nabla g \Vert_{\infty} \left\vert n^{1-d} \sum_{i = 1}^d\sum_{z \in \nZd}\sum_{x \in K^{(n)} : \vert x-z \vert \le \frac{C}{n}} H^{(n)}_i(x,z)\right\vert,
\end{align*}
where $c_1 > 0$ depends only on $C$ and we used the definition of $\nabla^{(n)}_i u_n(z)$, and the mean value theorem for $g(x)-g(z)$, which is applicable since $x \in K^{(n)} : \vert x-z \vert \le C/n$ is close to $z$. (Note that the $n$ from $\nabla^{(n)}_i$ cancels with the $n^{-1}$ from $(g(x)-g(z)) \le Cn^{-1}\Vert \nabla g \Vert_{\infty}$). Finally:
\begin{align*}
&\left\vert n^{1-d}\sum_{i = 1}^d\sum_{z \in \nZd}\sum_{x \in K^{(n)} : \vert x-z \vert \le \frac{C}{n}} H^{(n)}_i(x,z)\right\vert\\
&= \left\vert n^{1-d} \sum_{i = 1}^d\sum_{z \in \nZd}\sum_{x \in K^{(n)} : \vert x-z \vert \le \frac{C}{n}}\sum_{y \in \nZd : \vert x-y \vert \le \frac{C}{n}} \hspace{-0.6cm} \left(P^{x,y}(z+e_i/n,z)-P^{x,y}(z,z+e_i/n) \right)\Cna(x,y)\right\vert\\
&\le n^{2-d}\sum_{x \in K^{(n)}}\sum_{y \in \nZd : \vert x-y \vert \le \frac{C}{n}} \vert x-y \vert \vert \Cna(x,y) \vert\\
&\le c_2 \left[n^{-d}\sum_{x \in K^{(n)}} \left(n^2\sum_{y: \vert x-y \vert \le \frac{C}{n}} \vert x-y \vert^2 \vert J^{(n)}(x,y) \vert\right)^{\frac{2\theta}{2\theta-1}}\right]^{\frac{2\theta-1}{2\theta}} \left\Vert n^2 \sum_{y: \vert \cdot-y \vert \le \frac{C}{n}} \frac{\vert \Cna(\cdot,y)\vert^2}{J^{(n)}(\cdot,y)} \right\Vert_{L^{\theta}(K^{(n)})}^{\frac{1}{2}}\\
&\le c_3 \mu^{(n)}(K^{(n)})^{\frac{2\theta-1}{2\theta}}
\end{align*}
for some $c_2, c_3 > 0$, where we used \eqref{eq:BKUhelp}, \eqref{eq:BKUA3} and \eqref{K1}. Note that the resulting term in the previous estimate is bounded uniformly in $n$ due to the fact that $\mu^{(n)}(K^{(n)}) \to \vert \supp(g) \vert < \infty$. 
Recall that due to (i) in \autoref{thm:tightness}, the family $(u_n)$ is equicontinuous and therefore $\sup_{\vert z-z'\vert \le \frac{1}{n}} \vert u_n(z) - u_n(z') \vert \to 0$ as $n \to \infty$. Consequently, $\vert I^{(n)}_2\vert \to 0$, as desired.\\
Now, we consider $I^{(n)}_1$. We write
\begin{align*}
I^{(n)}_1 = 2n^{-d} \sum_{i = 1}^d\sum_{z \in \nZd} B^{(n)}_i(z)\nabla^{(n)}_i u_n(z)g(z) = 2 \sum_{i=1}^d \int_{\R^d} B^{(n)}_i(z)\nabla^{(n)}_i E^{(n)}u_n(z)E^{(n)}R^{(n)}g(z) \d z.
\end{align*}
There exists a subsequence of $(E^{(n)}\nabla^{(n)}_i u_n(z))_n$ converging weakly in $L^2(\R^d)$ to $\partial_i H$. This follows from the proof of (a) in \cite{BKU10}. Moreover, we can extract a further subsequence along which $B^{(n)}_i \to b_i$ boundedly and pointwise almost everywhere and also $E^{(n)}R^{(n)}g \to g$ uniformly on compacts. Along this further subsequence, we have that
\begin{align*}
I^{(n)}_1 \to 2\int_{\R^d} b_i(z)\partial_i H (z) g(z) \d z = \cE^{b_i}(H,g).
\end{align*}
We have shown that properties (a), (b) of \autoref{thm:convergence} hold, which yields the desired result.
\end{proof}

\autoref{thm:CLTl} answers question (i) from the beginning for the case $\alpha = 2$ and bounded range conductances $\Cn$. Now, we want to address question (ii), namely provide explicit conductances $\Cn$ that satisfy the assumptions of \autoref{thm:CLTl} and approximate a process that corresponds to a previously given form of type \eqref{eq:limitforml}. The following example contains an explicit computation of $B^{(n)}_i$ in a special case:

\begin{example}
Let us assume that we are in the special situation that for a given $\Cn$, we have that $\Cna$ is of the following form:
\begin{align*}
\Cna(x,y) = \beta^{(n)}(x,y)NN(x,y),
\end{align*}
and $\beta^{(n)} : \nZd \times \nZd \to \R$ satisfies $\beta^{(n)}(x,y) = -\beta^{(n)}(y,x)$ for every $x,y \in \nZd$. In that case, it holds for $z \in \nZd$:
\begin{align}
\label{eq:Bnspecial}
\begin{split}
B^{(n)}_i(z) &= n\left(H^{(n)}(z+e_i/n,z) + H^{(n)}_i(z,z)\right)\\
&= n\left(P^{z+e_i/n,z}(z+e_i/n,z)\Cna(z+e_i/n,z)-P^{z,z+e_i/n}(z,z+e_i/n)\Cna(z,z+e_i/n)\right)\\
&= 2n\Cna(z+e_i/n,z)\\
&= n \beta^{(n)}(z+e_i/n,z).
\end{split}
\end{align}
\end{example}

\begin{theorem}[concrete approximation]
\label{thm:approxl}
Let $a_{i,j} : \R^d \to \R$ be uniformly elliptic and bounded and let $b_i : \R^d \to \R$ with $|b_i|^2 \in L^{\theta}(\R^d)$ for some $\theta \in (\frac{d}{2},\infty]$. Let $(\cE, H^1(\R^d))$ defined in \eqref{eq:limitforml} be the associated lower-bounded semi-Dirichlet form and $X$ be the associated Markov process.\\
Then, there exists a sequence $(\Cn)$ satisfying \autoref{ass:local} with $a_{i,j}$, $b_i$ such that \eqref{K1}, \eqref{K2}, \eqref{CTail}, \eqref{Poinc}, \eqref{Sob} hold true with $\alpha =2$ and $\theta$. Furthermore, $\Xn$ is comparable to the NNRW in the sense of \eqref{eq:NNRWcomp} and of bounded range.\\
As a consequence, for each $T > 0$,  $x \in \R^d$ and $(x_n) \subset \nZd$ with $x_n \to x$ the $\PP^{x_n}$-laws of the continuous time Markov chain $(\Xn_t)_{t \in [0,T]}$ corresponding to $(\cEn,L^2(\nZd))$ weakly converge to the $\PP^x$-law of $(X_t)_{t \in [0,T]}$ with respect to the topology of $D([0,T];\R^d)$.
\end{theorem}

\begin{proof}
The proof can be seen as an extension of Theorem 5.5 in \cite{DeKu13} in the sense that we allow for the existence of an additional first order drift term. The idea of proof is the same.\\
We define $r_n = \lfloor n^{1-\beta}\rfloor/n \in n^{-1}\Z_{+}$ for some $\beta \in (0,1)$. Then $r_n \Z^d \subset \nZd$. For $x_0 \in r_n\Z^d$ we define the cube
\begin{align*}
Q(x_0,r_n) = \lbrace y \in \R^d : 0 \le \min_i (y_i - (x_0)_i) \le \max_i (y_i-(x_0)_i) < r_n \rbrace.
\end{align*}
Clearly, the set of all cubes $Q(x_0,r_n)$ covers the full space $\R^d$ and all cubes are pairwise disjoint.  Then, we set for $y \in \R^d$
\begin{align*}
&a^{(n)}_{i,j}(y) = \sum_{x_0 \in r_n \Z^d} \left(\dashint_{Q(x_0,r_n)}a_{i,j}(z) \d z \right) \mathbbm{1}_{Q(x_0,r_n)}(y),\\
&b^{(n)}_{i}(y) = \sum_{x_0 \in r_n \Z^d} \left(\dashint_{Q(x_0,r_n)}b_{i}(z) \d z \right) \mathbbm{1}_{Q(x_0,r_n)}(y).
\end{align*}
Moreover, $(a^{(n)}_{i,j}(y))$ is uniformly elliptic and bounded by $\Vert a_{i,j}\Vert_{\infty}$. By Jensen's inequality and the construction of the sets $Q(x_0,r_n)$:
\begin{align}
\label{eq:binnorm}
\begin{split}
\Vert|b^{(n)}_i|^2 \Vert_{L^{\theta}(\R^d)}^{\theta} &= \sum_{x_0 \in r_n \Z^d} \int_{Q(x_0,r_n)} \left(\dashint_{Q(x_0,r_n)}b_{i}(z) \d z \right)^{2\theta} \d y\\
&\le \sum_{x_0 \in r_n \Z^d} \dashint_{Q(x_0,r_n)} \int_{Q(x_0,r_n)} |b_{i}(z)|^{2\theta} \d z \d y \\
&= \Vert b_i^2 \Vert_{L^{\theta}(\R^d)}^{\theta}.
\end{split}
\end{align}
Therefore
$\Vert a^{(n)}_{i,j} - a_{i,j}\Vert_{L^1_{loc}(\R^d)} \to 0$, $\Vert b^{(n)}_i - b_i \Vert_{L^1_{loc}(\R^d)} \to 0$. Furthermore, $a^{(n)}_{i,j},b^{(n)}_i$ are piecewise constant functions in the sense that $a^{(n)}_{i,j}(x) = a^{(n)}_{i,j}([x]_{r_n})$, $b^{(n)}_i(x) = b^{(n)}_i([x]_{r_n})$ for every $x \in \nZd$. Note that it remains to find $\Cn$ such that $\Vert F^{(n)}_{i,j} - a^{(n)}_{i,j}\Vert_{L^1_{loc}(\R^d)} \to 0$ and $\Vert B^{(n)}_i -b^{(n)}_i\Vert_{L^1_{loc}(\R^d)} \to 0$ in order to verify \autoref{ass:local}.

Let us now define $\Cns$ according to the procedure laid out in \cite{DeKu13} such that $\Vert F^{(n)}_{i,j} - a^{(n)}_{i,j}\Vert_{L^1_{loc}(\R^d)} \to 0$ holds true. Note that from the construction in \cite{DeKu13} (Theorem 5.5), it follows that \eqref{CTail} is satisfied with $\alpha = 2$, $\Cns$ is of bounded range ($C = 2$) and there exists $\eps > 0$ such that for every $n \in \N$: $\Cns(x,y) \ge \eps$ if $\vert x-y \vert = \frac{1}{n}$. Since \eqref{Poinc} and \eqref{Sob} are satisfied for $\cE^{(n)}_{NN}$ (see \cite{Kum14}, \cite{Bar17}), it follows by \eqref{eq:NNRWcomp} that \eqref{Poinc} and \eqref{Sob} also hold true for $\cEns$ with $\alpha = 2, \sigma = \frac{1}{2n}$. 

It remains to find an antisymmetric sequence of conductances $\Cna$ such that the kernel $\Cn = \Cns + \Cna$ satisfies assumptions \eqref{K1}, \eqref{K2} and $\Vert B^{(n)}_i -b^{(n)}_i\Vert_{L^1_{loc}(\R^d)} \to 0$. We make the choice
\begin{align*}
\Cna(x,y) = \beta^{(n)}(x,y)NN(x,y)\mathbbm{1}_{\lbrace\vert\beta^{(n)}(x,y)\vert \le \eps \rbrace}(x,y),
\end{align*}
and $\beta^{(n)} : \nZd \times \nZd \to \R$ satisfies $\beta^{(n)}(x,y) = -\beta^{(n)}(y,x)$ for every $x,y \in \nZd$. According to \eqref{eq:Bnspecial}:
\begin{align*}
B^{(n)}_i(z) = n \beta^{(n)}(z+e_i/n,z)\mathbbm{1}_{\lbrace\vert\beta^{(n)}(z+e_i/n,z)\vert \le \eps \rbrace}(z+e_i/n,z), ~~z \in \nZd.
\end{align*}
We define 
\begin{align*}
\beta^{(n)}(x,y) =\begin{cases}
b^{(n)}_i(x)/n, ~~&\text{if } \exists x_0 \in \nZd ~\exists i \in \{1,\dots,d\} : x,y \in Q(x_0,r_n), x = y+e_i/n,\\
-b^{(n)}_i(x)/n, ~~&\text{if } \exists x_0 \in \nZd ~\exists i \in \{1,\dots,d\} : x,y \in Q(x_0,r_n), y = x+e_i/n,\\ 
0, ~~ &\text{else.}
\end{cases}
\end{align*}
Note that since $b^{(n)}_i$ is piecewise constant, $\beta^{(n)}$ is indeed antisymmetric. Furthermore, it follows:
\begin{align*}
B^{(n)}_i(z) = \begin{cases}
b^{(n)}_i(z)\mathbbm{1}_{\lbrace \vert b_i^{(n)}(z)\vert \le n\eps \rbrace}(z), ~~&\text{if } z,z+e_i/n \in Q(x_0,r_n) \text{ for some } x_0 \in r_n\Z^d,\\
0, ~~&\text{else.}
\end{cases}
\end{align*}
Note that by definition of $b_i^{(n)}$ and since $|b_i|^2 \in L^{\theta}(\R^d)$, there is $n_0 \in \N$ such that for every $n \ge n_0$ it holds that $\Vert B^{(n)}_i\Vert_{\infty} \le n\eps$ for every $i$.
We define $Q_i(x_0,r_n) = \lbrace z \in \nZd : z,z+e_i/n \in Q(x_0,r_n)\rbrace$ and $BQ_i(x_0,r_n) = Q(x_0,r_n) \setminus Q_i(x_0,r_n)$. Note that $\vert BQ_i(x_0,r_n) \vert = r_n^{d-1}n^{-1}$.
Consequently, for each compact set $K \subset \R^d$ and $n \ge n_0$ it holds by \eqref{eq:binnorm}:
\begin{align*}
\int_{K} \vert B^{(n)}_i(z) - b^{(n)}_i(z)\vert \d z &\le \Vert b^{(n)}_i\Vert_{L^{2\theta}(K)} \left\vert \bigcup_{x_0 \in r_n\Z^d : Q_i(x_0,r_n) \subset K}BQ_i(x_0,r_n)\right\vert^{1-\frac{1}{2\theta}}\\
&\le \Vert b_i \Vert_{L^{2\theta}(K)}\left[c(K) r_n^{-d}(r_n^{d-1}n^{-1})\right]^{1 - \frac{1}{2\theta}}\\
&= \Vert b_i^2 \Vert_{L^{\theta}(K)}^{\frac{1}{2}} \left[\frac{c(K)}{\lfloor n^{1-\beta}\rfloor}\right]^{1-\frac{1}{2\theta}} \to 0,
\end{align*}
where we used that the number of rectangles $Q_i(x_0,r_n)$ that are contained in $K$ can be bounded by $c(K)r_n^{-d}$, where $c(K) > 0$ is a constant depending on the diameter of $K$, only. It follows that $\Vert B^{(n)}_i - b^{(n)}_i\Vert_{L^1_{loc}(\R^d)} \to 0$, as desired.\\
Furthermore \eqref{K2} is satisfied by construction with $D = \frac{1}{2}$. In order to prove that \eqref{K1} holds true with $\theta$, we estimate for $n \in \N$, $x \in \nZd$:
\begin{align*}
n^2 \sum_{y \in \nZd} \frac{\vert \Cna(x,y)\vert^2}{\Cns(x,y)} &= n^2 \sum_{y \in \nZd : \vert x-y \vert = \frac{1}{n}} \frac{\vert \Cna(x,y)\vert^2}{\Cns(x,y)}\\
&\le  n^2 \sum_{y \in \nZd : \vert x-y \vert = \frac{1}{n}} \frac{\frac{1}{2n^2}| b_i^{(n)}(x)|^2}{\eps} \le c(\eps) | b_i^{(n)}(x)|^2
\end{align*}
for some $c = c(\eps) > 0$. This implies
\begin{align*}
\left\Vert n^2 \sum_{y \in \nZd} \frac{\vert \Cna(\cdot,y)\vert^2}{\Cns(\cdot,y)} \right\Vert_{L^{\theta}(\nZd)} \le c(\eps) \Vert (b_i^{(n)})^2 \Vert_{L^{\theta}(\nZd)} \le c(\eps) \Vert (b_i)^2 \Vert_{L^{\theta}(\R^d)} \lfloor n^{1-\beta}\rfloor^{-\frac{d}{\theta}}.
\end{align*}
This proves the desired result.
\end{proof}

Finally, we present a much simpler direct proof of \autoref{thm:approxl} in the case that the drift $b$ is of the form $b(x) = \nabla V(x)$ for a suitable function $V \in C^1(\R^d)$.

\begin{example}[concrete approximation: $b = \nabla V$]
Assume that $b = \nabla V$ for some $V \in C^1(\R^d) \cap L^{2\theta}(\R^d)$ for some $\theta \in (\frac{d}{2},\infty]$ and $a_{i,j}$ as before. Choose again $\Cns$ as in \cite{DeKu13} and let $\eps > 0$ be such that $\Cns(x,y) \ge \eps$ for every $x,y \in \nZd$ with $\vert x-y \vert = 1/n$. Now, we can set
\begin{align*}
\Cna(x,y) = (V(x)-V(y))NN(x,y)\mathbbm{1}_{\lbrace \vert V(x)-V(y) \vert \le \eps \rbrace}(x,y),
\end{align*}
and define $\Cn = \Cns + \Cna$. We claim that this choice of conductances $\Cn$ gives rise to a form $\cEn$ with associated Markov chains $(\Xn_t)_{t \in [0,T]}$ whose $\PP^{x_n}$-laws weakly converge towards the $\PP^x$-law of $(X_t)_{t \in [0,T]}$, which is associated to $\cE$ given as in \eqref{eq:limitforml} for every $(x_n) \subset \nZd$, $x \in \R^d$ with $x_n \to x$ and every $T > 0$.\\
In fact, one way to see this is to apply \autoref{thm:CLTl} as a black box, using the same line of arguments as in the proof of \autoref{thm:approxl} to verify the assumptions. Indeed for large enough $n$, according to \eqref{eq:Bnspecial}:
\begin{align*}
B^{(n)}_i(z) = n(V(z+e_i/n)-V(z)) = \nabla^{(n)}_i V(z),~~ z \in \nZd,
\end{align*}
from where one can directly read off that $\Cna$ satisfies \autoref{ass:local}. \eqref{K1}, \eqref{K2}, \eqref{CTail}, \eqref{Poinc}, \eqref{Sob} are immediate.\\
However, for this special case one can conveniently reprove \autoref{thm:CLTl}, simplifying some arguments along the way. Let us demonstrate how to obtain \eqref{eq:CLTlhelp1}. We compute for large enough $n$:
\begin{align*}
\cEna(u_n,R^{(n)} g) &= 2n^{2-d} \sum_{x \in \nZd} \sum_{y \in \nZd} (u_n(x)-u_n(y))g(x)(V(x)-V(y))NN(x,y)\\
&=n^{-d} \sum_{i = 1}^d \nabla^{(n)}_i u_n(x)g(x)\nabla^{(n)}_i V(x)\\
&+ n^{-d} \sum_{i = 1}^d \nabla^{(n)}_i u_n(x-e_i/n)g(x)\nabla^{(n)}_i V(x-e_i/n)\\
& \to 2 \int_{\R^d} \partial_i H(x) g(x) \partial_i V(x) \d x.
\end{align*}
The convergence can be justified by similar arguments as in the proof of \autoref{thm:CLTl}.
\end{example}

\subsection{Approximation of nonlocal forms}
In this section, we assume that \eqref{K1}, \eqref{K2}, \eqref{Poinc}, \eqref{Sob}, \eqref{CTail}, \eqref{CTail2} hold true with $\alpha \in (0,2)$, $\sigma > 0$ and $\theta \in (\frac{d}{\alpha},\infty]$. Moreover, we assume that
\begin{align}
\label{eq:extraassnl}
\sup_{n \in \N} \left\Vert n^{\alpha} \sum_{y \in \nZd : \vert \cdot-y \vert < r} \frac{\vert\Cna(\cdot,y)\vert^2}{J^{(n)}(\cdot,y)} \right\Vert_{L^{\theta}(\Omega \cap \nZd)} \to 0, \text{ as } r \searrow 0, 
\end{align}
for every compact set $\Omega \subset \R^d$. We will demonstrate at the end of this section that these assumptions are naturally satisfied for a huge class of conductances by providing some examples.\\
In this case, the limiting form will be governed by a jumping kernel $K : \R^d \times \R^d \to [0,\infty]$ and therefore be of nonlocal type. $K$ will be determined through the following property:

\begin{assumption}
\label{ass:nonlocal}
There exists $K : \R^d \times \R^d \to [0,\infty]$ such that
\begin{align}
\label{eq:alterass}
\int_{\mathcal{K}} f(x,y) n^{d+\alpha}\Cn([x]_n,[y]_n) \d y \d x \to \int_{\mathcal{K}} f(x,y)K(x,y) \d y \d x
\end{align}
for every compact set $\mathcal{K} \subset \R^d \times \R^d \setminus \diag$, $f \in C(\mathcal{K})$.
Moreover $K$ satisfies
\begin{align}
\label{eq:limitformKass}
&\sup_{x \in \R^d} \int_{\R^d} (1 \wedge \vert x-y \vert^2) K_s(x,y) \d y < \infty,~~  [v]_{H^{\alpha/2}(\R^d)}^2 \le C \cE^{K_s}(v,v), ~~ \forall v \in L^2(\R^d),\\
\label{eq:limitformKass2}
&\left\Vert \int_{\R^d} \frac{\vert K_a(\cdot,y)\vert^2}{J(\cdot,y)} \d y \right\Vert_{L^{\theta}(\R^d)} < \infty, ~~ \cE^J_B(v,v) \le C \cE^{K_s}_B(v,v), ~~ \forall v \in L^2(B), ~~ 
\end{align}
for some $C > 0$ and some symmetric jumping kernel $J$ and for every ball $B \subset \R^d$.
\end{assumption}

In analogy with the discrete case, we set $K_s(x,y)=\frac{1}{2}(K(x,y)+K(y,x))$, $K_a(x,y)=\frac{1}{2}(K(x,y)-K(y,x))$. Note that \autoref{ass:nonlocal} is sufficient for $(\cE,V(\R^d|\R^d))$ defined by
\begin{align}
\label{eq:limitformnl}
\cE(f,g) = 2\int_{\R^d} \int_{\R^d} (f(x)-f(y))g(x) K(x,y) \d y \d x
\end{align}
to be a regular lower-bounded semi-Dirichlet form on $L^2(\R^d)$ with
\begin{align*}
V(\R^d|\R^d) = \left\lbrace v \in L^2(\R^d) : (v(x)-v(y))K_s^{1/2}(x,y) \in L^2(\R^d \times \R^d)\right\rbrace
\end{align*}
and having $C_c^{lip}(\R^d)$ as a core (see \cite{ScWa15}). Moreover, \eqref{eq:limitformKass2} is a continuous version of \eqref{K1}.

\begin{remark}
Note that \eqref{eq:alterass} implies that for every sequence $(f_n),f \subset C(\mathcal{K})$ with $f_n \to f$ uniformly on $\mathcal{K}$:
\begin{align*}
\int_{\mathcal{K}} f_n(x,y) n^{d+\alpha}\Cn([x]_n,[y]_n) \d y \d x \to \int_{\mathcal{K}} f(x,y)K(x,y) \d y \d x.
\end{align*}
\end{remark}

We present the following analog of \autoref{thm:CLTl} in the case $\alpha \in (0,2)$:

\begin{theorem}[central limit theorem]
\label{thm:CLTnl}
Assume that \eqref{K1}, \eqref{K2}, \eqref{Poinc}, \eqref{Sob}, \eqref{CTail}, \eqref{CTail2} hold true with $\alpha \in (0,2)$, $\sigma > 0$ and $\theta \in (\frac{d}{\alpha},\infty]$. Moreover, assume that there is $K$ such that \autoref{ass:nonlocal} holds true. Let $(x_n) \subset \nZd$, $x \in \R^d$ with $x_n \to x$ and $T > 0$. Then, the $\PP^{x_{n}}$-laws of $(\Xn_t)_{t \in [0,T]}$ converge weakly, with respect to the $D([0,T];\R^d)$-topology to the $\PP^x$-law of $(X_t)_{t \in [0,T]}$, where $(X,\PP^x)$ is the Markov process corresponding to the form $(\cE,V(\R^d|\R^d))$.
\end{theorem}

\begin{proof}
The proof uses the same ideas as in \cite{HuKa07} and \cite{BKU10}. We need to verify properties (a), (b) from \autoref{thm:convergence}. Let $f \in C_c(\R^d)$, $g \in C_c^{lip}(\R^d)$, $(n_j) \subset \N$. We denote $u_{n_j} = U^{(n_j)}_{\lambda} R^{(n_j)} f$ and $H = \lim_{j \to \infty} E^{(n_j)} u_{n_j}$ (see  \autoref{thm:tightness} (i)), where the convergence holds true uniformly on compacts. As in the proof of \autoref{thm:CLTl}, we assume for simplicity that already $(E^{(n)} u_n)_n$ converges to $H$.\\
First, we infer from G\r{a}rding's inequality for $\cEn$ (see \eqref{eq:Garding-markov}) and \eqref{eq:resolventestimate} that the family $(\cEns(u_n,u_n))_n$ is uniformly bounded. In order to prove that $H \in V(\R^d|\R^d)$, we proceed similar to \cite{BKU10}. By \autoref{ass:nonlocal} and equicontinuity of $(u_n)_n$ (see \autoref{cor:regest}), as well as the convergence $E^{(n)}u_n \to H$, we obtain for every compact set $\Omega \subset \R^d$ 
\begin{align*}
\int_{\Omega} &\int_{\Omega \cap \{y : N^{-1} \le\vert x-y \vert \le N\}} (H(x)-H(y))^2 K_s(x,y)\d y \d x\\
&= \int_{\Omega} \int_{\Omega \cap \{y : N^{-1} \le\vert x-y \vert \le N\}} (H(x)-H(y))^2 K(x,y)\d y \d x\\
&\le \lim_{n \to \infty} n^{\alpha-d} \sum_{x \in \nZd} \sum_{y \in \nZd : N^{-1} \le \vert x-y \vert \le N} (u_n(x)-u_n(y))^2 \Cn(x,y)\\
&= \limsup_{n \to \infty} n^{\alpha-d} \sum_{x \in \nZd} \sum_{y \in \nZd : N^{-1} \le \vert x-y \vert \le N} (u_n(x)-u_n(y))^2 \Cns(x,y)\\
&\le \limsup_{n \to \infty} \cEns(u_n,u_n) < \infty.
\end{align*} 
Therefore, by taking $N \nearrow \infty$, and approximating $\Omega \nearrow \R^d$, we obtain that $\cEs(H,H) < \infty$. As $H \in L^2(\R^d)$ is an immediate consequence of the convergence $E^{(n)}u_n \to H$ and \eqref{eq:resolventestimate}, we infer that indeed $H \in V(\R^d|\R^d)$.

It remains to prove that there exists a subsequence $(n_k) \subset \N$ such that $\cE^{(n_k)}(u_{n_k},R^{(n_k)} g) \to \cE(H,g)$. We observe that for every $N > 1$:
\begin{align*}
2n^{\alpha-d} \sum_{x \in \nZd} &\sum_{y \in \nZd : N^{-1 \le }\vert x-y \vert \le N} (u_n(x)-u_n(y))g(x)\Cn(x,y)\\
&\to \int_{\R^d}\int_{\lbrace N^{-1} \le \vert x-y \vert \le N \rbrace} (H(x)-H(y))g(x) K(x,y) \d y \d x,
\end{align*}
which is due to \autoref{ass:nonlocal} and the fact that $(u_n)$ is equicontinuous and equibounded by (i) of \autoref{thm:tightness} and converges to $H$ uniformly on compacts.\\
It remains to show that the quantities
\begin{align}
\label{eq:CLTnlhelp0}
&\int_{\R^d}\int_{\R^d \cap \{y : \vert x-y \vert \not\in [N^{-1},N]\}} (H(x)-H(y))g(x)K(x,y) \d y \d x,\\
\label{eq:CLTnlhelp1}
&2n^{\alpha-d} \sum_{x \in \nZd} \sum_{y \in \nZd :\vert x-y \vert \not\in [N^{-1},N]} (u_n(x)-u_n(y))g(x)\Cn(x,y)
\end{align}
can be made arbitrarily small by choosing $N$ large enough. For \eqref{eq:CLTnlhelp0} this immediately follows from $\cEs(H,H) < \infty$ and \eqref{eq:limitformKass2}. In order to estimate \eqref{eq:CLTnlhelp1}, we denote  $K^{(n)} = \supp(g) \cap \nZd$ and let $M > 0$ such that $\{x \in \nZd : \dist(x, K^{(n)}) < N^{-1}\} \subset \Bn_M$. On the one hand,
\begin{align*}
2n^{\alpha-d} &\sum_{x \in K^{(n)}} \sum_{y \in \nZd :\vert x-y \vert > N} (u_n(x)-u_n(y))g(x)\Cn(x,y)\\
&\le c_1\cEns(u_n,u_n)^{1/2} \Vert g \Vert_{L^2(K^{(n)})} \left( n^{\alpha} \sup_{x \in K^{(n)}}\sum_{y \in \nZd : \vert x-y \vert > N} \Cn(x,y) \right)^{1/2}\\
&\le c_2 N^{-\delta/2} \cEns(u_n,u_n)^{1/2} \Vert g \Vert_{L^2(K^{(n)})},
\end{align*}
where $c_1,c_2 > 0$ are constants, and we used \eqref{CTail2} in the last step. Note that the quantities $\cEns(u_n,u_n)$ and $\Vert g \Vert_{L^2(K^{(n)})}$ are uniformly bounded in $n$. One the other hand,
\begin{align*}
2n^{\alpha-d} &\sum_{x \in \nZd} \sum_{y:\vert x-y \vert < N^{-1}} (u_n(x)-u_n(y))g(x)\Cns(x,y)\\
&=n^{\alpha-d} \sum_{x \in \nZd} \sum_{y:\vert x-y \vert < N^{-1}} (u_n(x)-u_n(y))(g(x)-g(y))\Cns(x,y)\\
&\le c_3 \cEns(u_n,u_n)^{\frac{1}{2}} \left(n^{\alpha-d} \sum_{x \in \Bn_M} \sum_{y:\vert x-y \vert < N^{-1}} (g(x)-g(y))^2\Cns(x,y)\right)^{\frac{1}{2}}\\
&\le c_4 \cEns(u_n,u_n)^{\frac{1}{2}} \Vert \nabla g\Vert_{L^\infty(\R^d)} \mu^{(n)}(\Bn_M)^{\frac{1}{2}} \left( n^{\alpha} \hspace{-0.1cm} \sup_{x \in \Bn_M} \sum_{y:\vert x-y \vert < N^{-1}} \hspace{-0.4cm}\vert x-y \vert^2 \Cns(x,y)\right)^{\frac{1}{2}}\\
&\le c_5 N^{(\alpha-2)/2} \cEns(u_n,u_n)^{\frac{1}{2}} \Vert \nabla g\Vert_{L^\infty(\R^d)} \mu^{(n)}(\Bn_M)^{\frac{1}{2}},
\end{align*}
where $c_3,c_4,c_5 > 0$ are constants, and we applied \eqref{eq:BKUA3} in the last step. Again, this quantity can be made arbitrarily small by choosing $N$ large since $\cEns(u_n,u_n)$ and $\mu^{(n)}(\Bn_M)$ are bounded in $n$. Finally, the contribution of $\Cna$ to \eqref{eq:CLTnlhelp1} can be estimated as follows:
\begin{align*}
2n^{\alpha-d} &\sum_{x \in \nZd} \sum_{y:\vert x-y \vert < N^{-1}} (u_n(x)-u_n(y))g(x)\Cna(x,y)\\
&\le c_6 \cE^{(n),J^{(n)}}(u_n,u_n)^{1/2} \left( n^{\alpha-d} \sum_{x \in K^{(n)}} \sum_{y:\vert x-y \vert < N^{-1}} g^2(x)\frac{\vert\Cna(x,y)\vert^2}{J^{(n)}(x,y)} \right)^{1/2} \\
&\le c_7 \cEns(u_n,u_n)^{1/2} \Vert g \Vert_{L^{2\theta'}(K^{(n)})} \left\Vert n^{\alpha} \sum_{y: \vert \cdot -y \vert < N^{-1}} \frac{\vert\Cna(\cdot,y)\vert^2}{J^{(n)}(\cdot,y)} \right\Vert_{L^{\theta}(K^{(n)})}^{1/2}
\end{align*}
for $c_6,c_7 > 0$. The right hand side can be made arbitrarily small by taking $N$ large due to \eqref{eq:extraassnl} and by uniform boundedness of $\cEns(u_n,u_n)$ and $\Vert g \Vert_{L^{2\theta'}(K^{(n)})}$ in $n$.\\
Consequently, the quantity in \eqref{eq:CLTnlhelp1} can be made arbitrarily small and we conclude the proof.
\end{proof}

We define $\vert h \vert_{\infty} = \max_{i \in \{1,\dots,d\}} \vert h_i \vert$.
As in the previous subsection, we provide a result on concrete approximation of a given process in the nonlocal case. These processes can be regarded as $\alpha$-stable like processes with a nonlocal drift term. We refer to \cite{KaWe22} for a more detailed discussion of the corresponding generator and its associated bilinear form. For simplicity, we consider only the special case, where $K(x,y) \asymp |x-y|^{-d-\alpha}$. However, note that an extension to more general kernels is straightforward provided that Sobolev -- and Poincar\'e inequalities hold true for $K_s$ and are inherited to $\Cns$.

\begin{theorem}[concrete approximation]
\label{thm:approxnl}
Let $\alpha \in (0,2)$, $\theta \in (\frac{d}{\alpha},\infty]$, and $K : \R^d \times \R^d \to [0,\infty]$ be such that there exist $\Lambda,C > 0$ with
\begin{align}
\label{eq:capproxnlass1}
&\Lambda^{-1} \vert x-y \vert^{-d-\alpha} \le K(x,y) \le \Lambda\vert x-y \vert^{-d-\alpha}, ~~ \forall x,y \in \R^d\\
\label{eq:capproxnlass2}
&\left\Vert\int_{\R^d} \frac{\vert K_a(\cdot,y)\vert^2}{K_s(\cdot,y)} \d y \right\Vert_{L^{\theta}(\R^d)} \le C, ~~ \left\Vert \int_{\R^d \cap \{y : \vert \cdot-y \vert < r\}} \frac{\vert K_a(\cdot,y)\vert^2}{K_s(\cdot,y)} \d y \right \Vert_{L^{\theta}(\R^d)} \to 0 \text{ as } r \searrow 0.
\end{align}
Let $(\cE,H^{\alpha/2}(\R^d))$ defined in \eqref{eq:limitformnl} be the associated regular lower-bounded semi-Dirichlet form on $L^2(\R^d)$ and $X$ be the associated Markov process. Then, there exists a sequence $(\Cn)_n$ satisfying \autoref{ass:nonlocal} with $K$ such that \eqref{K1}, \eqref{eq:extraassnl}, \eqref{K2}, \eqref{CTail}, \eqref{CTail2}, \eqref{Poinc}, \eqref{Sob} hold true with $\alpha$, $\sigma = 2\sqrt{d}$ and $\theta$.\\
As a consequence, for each $T > 0$,  $x \in \R^d$ and $(x_n) \subset \nZd$ with $x_n \to x$ the $\PP^{x_n}$-laws of the continuous time Markov chain $(\Xn_t)_{t \in [0,T]}$ corresponding to $(\cEn,L^2(\nZd))$ weakly converge to the $\PP^x$-law of $(X_t)_{t \in [0,T]}$ with respect to the topology of $D([0,T];\R^d)$.
\end{theorem}

\begin{proof}
We define for $x,y \in \nZd$:
\begin{align*}
\Cn(x,y) = 4^d n^{d-\alpha} \int_{\lbrace\vert x-w\vert_{\infty} < \frac{1}{2n}\rbrace}\int_{\lbrace\vert y-z\vert_{\infty} < \frac{1}{2n}\rbrace} K(w,z) \d w \d z \mathbbm{1}_{\{\vert x- y\vert_{\infty} \ge \frac{2}{n} \}}(x,y).
\end{align*}
It remains to verify assumptions \eqref{K1}, \eqref{K2}, \eqref{CTail}, \eqref{CTail2}, \eqref{Poinc}, \eqref{Sob}, \eqref{eq:extraassnl} and  \autoref{ass:nonlocal} for $\Cn$. Then the assertion follows from \autoref{thm:CLTnl}. \\
First, we denote $\Bn_{\infty}(x) := \{z \in \nZd : \vert x-z \vert_{\infty} < \frac{1}{2n}\}$ and observe that for $\vert x-y \vert_{\infty} > \frac{2}{n}$:
\begin{align*}
\Cns(x,y) &= 4^d n^{-d-\alpha} \dashint_{\Bn_{\infty}(x)} \dashint_{\Bn_{\infty}(y)} K_s(w,z) \d w \d z ,\\
\Cna(x,y) &= 4^d n^{-d-\alpha} \dashint_{\Bn_{\infty}(x)} \dashint_{\Bn_{\infty}(y)} K_a(w,z) \d w \d z. 
\end{align*}
First, we easily see that by \eqref{eq:capproxnlass1} there exist $c_1,c_2 > 0$:
\begin{align*}
c_1 \vert nx-ny \vert^{-d-\alpha} \le \Cns(x,y) \le c_2 \vert nx-ny \vert^{-d-\alpha}, ~~\forall x,y \in \nZd : \vert x-y \vert_{\infty} > \frac{2}{n}.
\end{align*}
From here, one can prove that \eqref{Poinc}, \eqref{Sob} hold true with $\sigma = 4\sqrt{d}$. Moreover, \eqref{CTail}, \eqref{CTail2} are immediate upon noticing that $\Cns(x,y) = 0$, whenever $\vert x-y\vert_{\infty} \le \frac{2}{n}$.\\
\eqref{K1} can be proved as Proposition 2.14 in \cite{MSS18}, estimating for every $x \in \nZd$:
\begin{align*}
n^{\alpha} \sum_{y \in \nZd} \frac{\vert \Cna(x,y)\vert^2}{\Cns(x,y)} &\le  n^{-d} \sum_{y \in \nZd} \frac{\left( \dashint_{\Bn_{\infty}(x)} \dashint_{\Bn_{\infty}(y)} K_a(w,z) \d w \d z \right)^2}{\left( \dashint_{\Bn_{\infty}(x)} \dashint_{\Bn_{\infty}(y)} K_s(w,z) \d w \d z \right)}\\
&\le c n^{-d}\sum_{y \in \nZd}\dashint_{\Bn_{\infty}(x)} \dashint_{\Bn_{\infty}(y)} \frac{\vert K_a(w,z)\vert^2}{K_s(w,z)} \d w \d z\\
&\le c \dashint_{\Bn_{\infty}(x)} \int_{\R^d} \frac{\vert K_a(w,z)\vert^2}{K_s(w,z)} \d w \d z.
\end{align*}
Here, we used \eqref{eq:capproxnlass2}. Therefore, by Jensen's inequality
\begin{align*}
n^{-d} \sum_{x \in \nZd} \left( n^{\alpha} \sum_{y \in \nZd} \frac{\vert \Cna(x,y)\vert^2}{\Cns(x,y)} \right)^{\theta} &\le c n^{-d} \sum_{x \in \nZd} \left( \dashint_{\Bn_{\infty}(x)} \int_{\R^d} \frac{\vert K_a(w,z)\vert^2}{K_s(w,z)} \d w \d z \right)^{\theta}\\
&\le c \int_{\R^d} \left( \int_{\R^d} \frac{\vert K_a(w,z)\vert^2}{K_s(w,z)} \d w \right)^{\theta} \d z,
\end{align*}
which yields \eqref{K1}.
Analogously, one verifies \eqref{eq:extraassnl}.
As assumption \eqref{K2} is immediate from \eqref{eq:capproxnlass1}, it remains to check \autoref{ass:nonlocal}. Note that by the proof of Theorem 2.3 in \cite{HuKa07} it follows that $\Cn \to K$ in $L^1(\mathcal{K})$ for every $\mathcal{K} \subset (\R^d \times \R^d) \setminus \text{diag}$ compact. Their arguments do not require symmetry of $\Cn$ or $K$. Therefore, \eqref{eq:alterass} holds true and \autoref{thm:CLTnl} is applicable.
\end{proof}

We end this section with an example, demonstrating what kind of conductances $(\Cn)$ are admissible for \autoref{thm:CLTnl}.

\begin{example}
Consider the sequence $(\Cn)$, defined by
\begin{align*}
&\Cns(x,y) = g^{(n)}_1(nx,ny)\vert nx-xy\vert^{-d-\alpha} = n^{-d-\alpha} g^{(n)}_1(nx,ny)\vert x-y\vert^{-d-\alpha},\\
&\Cna(x,y) = n^{-d-\alpha}g^{(n)}_2(nx,ny) \left(\vert x-y\vert^{-d-\beta}\mathbbm{1}_{\{\vert x-y \vert \le 1\}}(x,y) + \vert x-y\vert^{-d-\gamma}\mathbbm{1}_{\{\vert x-y \vert > 1\}}(x,y) \right),
\end{align*}
where $g^{(n)}_i : \Z^d \times \Z^d \to [-M_i,M_i]$ for some $M_i > 0$, $i \in \{1,2\}$, and $0 < 2\beta < \alpha < 2\gamma < 2$, such that $g^{(n)}_1(x,y) = g^{(n)}_1(y,x) \ge 0$ and $g^{(n)}_2(x,y) = -g^{(n)}_2(y,x)$ and $\Cn \ge 0$.\\
It is well-known that \eqref{Poinc}, \eqref{Sob}, \eqref{CTail}, \eqref{CTail2} hold. Moreover, \eqref{K1} is satisfied:
\begin{align*}
n^{\alpha} \hspace{-0.3cm}\sum_{y \in \nZd} \frac{\vert \Cna(x,y)\vert^2}{\Cns(x,y)} &\le  \frac{M_2^2}{M_1}\left( n^{2\beta-\alpha} \hspace{-0.3cm}\sum_{h \in \nZd : \vert h \vert \le 1} \hspace{-0.3cm}\vert nh\vert^{-d+(\alpha-2\beta)} + n^{2\gamma - \alpha} \hspace{-0.3cm}\sum_{h \in \nZd : \vert h \vert > 1} \hspace{-0.3cm}\vert nh\vert^{-d-(2\gamma - \alpha)} \right)\\
&\le \frac{c_1M_2^2}{M_1}\left( n^{2\beta-\alpha} \hspace{-0.2cm} \sum_{h \in \Z^d : \vert h \vert \le n} \hspace{-0.2cm} \vert h\vert^{-d+(\alpha-2\beta)} + n^{2\gamma - \alpha} \hspace{-0.2cm} \sum_{h \in \Z^d : \vert h \vert > n} \hspace{-0.2cm} \vert h\vert^{-d-(2\gamma - \alpha)} \right)\\
&\le c_2 < \infty,
\end{align*}
where $c_1,c_2 > 0$ are constants. Note that \eqref{eq:extraassnl} can be deduced from the following computation:
\begin{align*}
n^{\alpha} \sum_{y : \vert x-y \vert < r} \frac{\vert \Cna(x,y)\vert^2}{\Cns(x,y)} \le \frac{c_1M_2^2}{M_1} n^{2\beta-\alpha} \sum_{h \in \Z^d : \vert h \vert \le nr} \vert h\vert^{-d+(\alpha-2\beta)} \le c_3 r^{\alpha-2\beta}, ~~ 0 < r < 1,
\end{align*}
for some $c_3 > 0$.
\end{example}

\begin{example}
A similar computation as before yields \eqref{K1} for $\Cn$ given by
\begin{align*}
\Cns(x,y) &= g^{(n)}(nx,ny)\vert nx-ny \vert^{-d-\alpha},\\
\Cna(x,y) &= (V(x)-V(y))\vert nx-ny \vert^{-d-\alpha}\mathbbm{1}_{\{\vert x-y \vert \le 1\}}(x,y),
\end{align*}
where $g^{(n)} : \nZd \times \nZd \to [\lambda,\Lambda]$ for some $0 < \lambda \le \Lambda < \infty$, $V \in C^{\gamma}(\R^d)$, where $\gamma > \alpha/2$, and $\Cn \ge 0$. Indeed, write $\beta := \alpha-\gamma < \alpha/2$, and note that
\begin{align*}
\vert\Cna(x,y)\vert \le c\vert x-y \vert^{\gamma}\vert nx-ny \vert^{-d-\alpha}\mathbbm{1}_{\{\vert x-y \vert \le 1\}}(x,y) \le c n^{\beta-\alpha}\vert nx-ny \vert^{-d-\beta}\mathbbm{1}_{\{\vert x-y \vert \le 1\}}(x,y),
\end{align*}
where $c > 0$ is some constant. From here, one inserts the same computation as in the previous example and obtains \eqref{K1}.
\end{example}

\end{document}